\documentclass[11pt,a4paper,reqno]{article}
\usepackage[latin1]{inputenc}  
\usepackage[T1]{fontenc}
\usepackage[english]{babel}
\usepackage{epsfig,epsf,graphicx,xcolor}
\usepackage{amsfonts,amssymb,latexsym,a4wide,xspace}
\usepackage{amsmath,delarray,enumerate,calc,amsthm}
\usepackage{bbm}
\usepackage{txfonts}
\usepackage{paralist}
\usepackage{authblk}
\usepackage{tikz-cd}
\usepackage{stmaryrd}
\usepackage[mathscr]{euscript}

\newcommand{\id}{\operatorname{id}}
\newcommand{\supp}{\operatorname{supp}}

\newcommand{\cA}{{\mathcal A}}

\newcommand{\cB}{{\mathcal B}}
\newcommand{\cF}{{\mathcal F}}
\newcommand{\cG}{{\mathcal G}}
\newcommand{\cH}{{\mathcal H}}

\newcommand{\cK}{{\mathcal K}}

\newcommand{\cO}{{\mathcal O}}
\newcommand{\cP}{{\mathcal P}}

\newcommand{\cM}{{\mathcal M}}
\newcommand{\cMG}{\cM^{\scriptscriptstyle G}}
\newcommand{\cS}{{\mathcal S}}
\newcommand{\LL}{{\mathscr L}}
\newcommand{\BB}{{\mathscr B}}
\newcommand{\R}{{\mathbbm R}}
\newcommand{\N}{{\mathbbm N}}

\newcommand{\Z}{{\mathbbm Z}}

\newcommand{\T}{{\mathbbm T}}

\newcommand{\inn}{\operatorname{int}}
\newcommand{\card}{\operatorname{card}}

\newcommand{\hx}{{\hat{x}}}
\newcommand{\hy}{{\hat{y}}}
\newcommand{\hX}{{\hat{X}}}
\newcommand{\hXprime}{{\widehat{X'}}}

\newcommand{\hT}{\hat{T}}
\newcommand{\hTprime}{\widehat{T'}}

\newcommand{\nuW}{\nu_{\scriptscriptstyle W}}
\newcommand{\mystrut}{}
\newcommand{\nuWbar}{\nu_{\scriptscriptstyle \overline{W\mystrut}}}
\newcommand{\nuWprime}{\nu_{\scriptscriptstyle W'}}
\newcommand{\nuWreg}{\nu_{\scriptscriptstyle W_{reg}}}

\newcommand{\nuWG}{\nu_{\scriptscriptstyle W}^{\scriptscriptstyle G}}
\newcommand{\nuWGbar}{\nu_{\scriptscriptstyle \overline{W\mystrut}}^{\scriptscriptstyle G}}

\newcommand{\nuWGprime}{\nu_{\scriptscriptstyle W'}^{\scriptscriptstyle G}}
\newcommand{\nuWGzero}{\nu_{\scriptscriptstyle W_0}^{\scriptscriptstyle G}}
\newcommand{\GnuW}{{\cG\nuW}}

\newcommand{\MW}{\cM_{\scriptscriptstyle W}}
\newcommand{\MWbar}{\cM_{\scriptscriptstyle \overline{W\mystrut}}}
\newcommand{\MWGbar}{\cM_{\scriptscriptstyle \overline{W\mystrut}}^{\scriptscriptstyle G}}
\newcommand{\MWreg}{\cM_{\scriptscriptstyle W_{reg}}}
\newcommand{\MWprime}{\cM_{\scriptscriptstyle W'}}
\newcommand{\MWzero}{\cM_{\scriptscriptstyle W_0}}

\newcommand{\GMW}{\cG\cM_{\scriptscriptstyle W}}
\newcommand{\MWG}{\cM^{\scriptscriptstyle G}_{\scriptscriptstyle W}}
\newcommand{\MWGprime}{\cM^{\scriptscriptstyle G}_{\scriptscriptstyle W'}}
\newcommand{\MWGzero}{\cM^{\scriptscriptstyle G}_{\scriptscriptstyle W_0}}

\newcommand{\PG}{P^{\scriptscriptstyle G}}

\newcommand{\piG}{\pi^{\scriptscriptstyle G}}
\newcommand{\piH}{\pi^{\scriptscriptstyle H}}
\newcommand{\piHprime}{\pi^{\scriptscriptstyle H'}}
\newcommand{\piGH}{\pi^{\scriptscriptstyle G\times H}}
\newcommand{\pihX}{\pi^{\scriptscriptstyle \hX}}

\newcommand{\CW}{C_{\scriptscriptstyle W}}

\newcommand{\CWbar}{C_{\scriptscriptstyle \overline{W\mystrut}}}
\newcommand{\CWzero}{C_{\scriptscriptstyle W_0}}

\newcommand{\0}{\underline{0}}
\newcommand{\QM}{Q_{\scriptscriptstyle W}}

\newcommand{\QMreg}{Q_{\scriptscriptstyle W_{reg}}}
\newcommand{\QMG}{Q_{\scriptscriptstyle W}^{\scriptscriptstyle G}}
\newcommand{\QMGbar}{Q_{\scriptscriptstyle \overline{W}}^{\scriptscriptstyle G}}
\newcommand{\QMGprime}{Q_{\scriptscriptstyle W'}^{\scriptscriptstyle G}}

\newcommand{\oneone}{1-1\xspace}
\newcommand{\dL}{\mathrm{dens}(\LL)}
\newcommand{\nuG}{\nu^{\scriptscriptstyle G}}

\newcommand{\KW}{\cK_{\scriptscriptstyle W}}
\newcommand{\SH}{\cS_{\scriptscriptstyle H}}
\newcommand{\ocl}[1]{\overline{\cO(#1)}}

\newcommand{\WP}{W_P}
\newcommand{\Mmin}{\cM_{\scriptstyle min}}
\newcommand{\MminG}{\cM^{\scriptscriptstyle G}_{\scriptstyle min}}

\newtheorem{theoremI}{Theorem}

\newtheorem{theoremII}{Theorem}

\newtheorem{theoremIprime}{Theorem}

\newtheorem{theoremIIprime}{Theorem}

\newtheorem{fact}{Fact}

\newtheorem {definition}{Definition}[section] 
\newtheorem {lemma}[definition]{Lemma}
\newtheorem {bemerkung}[definition]{Remark}
\newtheorem{proposition}[definition]{Proposition}
\newtheorem {corollary}[definition]{Corollary}
\newtheorem{beispiel}[definition]{Example}
\newtheorem{frage}[definition]{Problem}
\newenvironment{remark} {\begin{bemerkung} \normalfont }{\end{bemerkung}}

\begin{document}

\title{Periods and factors of weak model sets}
\author{Gerhard Keller and Christoph Richard\;
\thanks{These notes profited enormously from the stimulating research atmosphere at the workshop ``Combining Aperiodic Order with Structural Disorder'' at the Lorentz Center, Leiden, 2016.}}
\affil{Department Mathematik, Universit\"at Erlangen-N\"urnberg}
\date{\today}


\maketitle


\begin{abstract}
There is a renewed interest in weak model sets due to their connection to $\cB$-free systems \cite{BKKL2015}, which emerged from Sarnak's program on the M\"obius disjointness conjecture. Here we continue our recent investigation \cite{KR2015} of the extended hull $\MWG$, a dynamical system naturally associated to a weak model set in an abelian group $G$ with relatively compact window $W$. For windows having a nowhere dense boundary (this includes compact windows), we identify the maximal equicontinuous factor of $\MWG$ and give a sufficient condition when $\MWG$ is an almost \oneone extension of its maximal equicontinuous factor. 
If the window is  measurable with positive Haar measure and is almost compact, then the system $\MWG$ equipped with its Mirsky measure is isomorphic to its Kronecker factor. For general nontrivial ergodic probability measures on $\MWG$, we provide a kind of lower bound for the Kronecker factor.
All relevant factor systems are natural $G$-actions on quotient subgroups of the torus underlying the weak model set. These are obtained by factoring out suitable window periods. Our results are specialised to  the usual hull of the weak model set, and they are also interpreted for $\cB$-free systems. 
\end{abstract}

\renewcommand{\thefootnote}{\fnsymbol{footnote}} 
\renewcommand{\thefootnote}{\arabic{footnote}} 

\section{Introduction}

Fix two  locally compact second countable abelian groups $G$ and $H$. Typically, $G=\Z^d$ or $\R^d$, whereas $H$ will often be a more general group. Take a cocompact lattice $\LL\subseteq G\times H$ in generic position, i.e., 
$\LL$ projects injectively to $G$ and densely to $H$. Consider a  relatively compact and measurable subset $W$ of $H$ which is called the window. A weak model set $\Lambda\subset G$ is obtained by projecting all lattice points inside the strip $G\times W$ to $G$. 
The resulting set $\Lambda=\piG(\LL\cap(G\times W))$ is also called ``cut-and-project set'', and $H$ is  called ``internal space''.  Any weak model set is  uniformly discrete. Model sets additionally require $\inn(W)\ne\emptyset$, resulting in a relatively dense point set. They have been introduced by Meyer \cite{Meyer70,Meyer72} within a harmonic analysis context and, surprisingly, turned out later to describe physical quasicrystals. By now there is an abundant literature on model sets, see e.g.~the list of references in \cite{BaakeGrimm13}. Weak model sets have been initially studied by Schreiber \cite{Schreiber71,Schreiber73}. Their name was coined by Moody \cite{Moody2002}, see \cite{HuckRichard15} for further background.

Dynamical systems techniques have 
turned out to be a powerful tool to analyse model sets \cite{FHK2002, Schlottmann00, Robinson2007, BLM07, LenzRichard07}. Here one considers the hull of a model set, i.e., its translation orbit closure with respect to a Hausdorff-type metric, and one seeks to infer properties of the model set from its hull. For example, pure point diffraction spectrum of a model set can be inferred from (and is in fact equivalent to) pure point dynamical spectrum of its hull, equipped with its pattern frequency measure, see e.g.~\cite[Thm.~7]{BaakeLenz2004}. For general so-called regular model sets, pure point dynamical spectrum of their hull was shown by Schlottmann \cite{Schlottmann00}\,\footnote{An earlier diffraction result by Hof \cite{Hof1} for regular model sets having Euclidean internal space relied on harmonic analysis techniques. This has been extended beyond the Euclidean setting only recently \cite{RS17}.}. It is the aim of this article to perform a dynamical analysis for general weak model sets having a compact or close-to-compact window, thereby refining recent results \cite{BaakeHuck2015, KR2015}.

A natural dynamical question concerns the relation of the model set hull to its maximal equicontinuous factor and to its Kronecker factor, equipped with the  pattern frequency measure. Since model sets inherently display a high degree of regularity, one is tempted to expect almost isomorphisms to these factors, under mild assumptions on the window. One might also expect that these factors are isomorphic to the ``torus'' $\hX=(G\times H)/\LL$.   In this context, previous topological standard assumptions on the window were compactness, topological regularity $W=\overline{\inn(W)}$  and aperiodicity, i.e., $h+W=W$ implies $h=0$, compare the discussion in \cite[Sec.~4.3]{KR2015}. Note that, in Euclidean internal space, any window is aperiodic. Given these assumptions, the hull is an almost one-to-one extension of its maximal equicontinuous factor. This was shown in \cite{Robinson2007} for so-called non-singular model sets. For measure-theoretic results, a previous additional assumption was almost vanishing boundary of the window, resulting in a uniquely ergodic hull. In that case, isomorphism to the Kronecker factor has been shown in \cite{Schlottmann00, Robinson2007}. In all cases, the relevant factor is indeed the underlying torus $\hX$, and a factor map is provided by the so-called torus parametrisation \cite{BHP97, BLM07}, i.e., a natural continuous map which assigns to any element of the hull its torus coordinate.

However, classic model sets such as the Fibonacci chain ($G=\R$ and $H=\mathbb R$, see \cite[Ex.~7.3]{BaakeGrimm13}) and their generalizations, the Sturmian chains, and also the discrete counterparts of these sets, namely Fibonacci and Sturmian sequences ($G=\Z$ and $H=\T$, see \cite[Ch.~6]{PytheasFogg})
are not subsumed by the above results, as their windows are half-open intervals. Compact internal spaces different from $\T^d$ arise for other model sets which are subsets of lattices. Early examples are squarefree integers and visible lattice points \cite{BMP00}. These are not subsumed by the above results either, as their compact aperiodic windows have no interior points. This results in point sets having arbitrarily large holes.
Other non-standard examples are non-regular Toeplitz sequences, having their odometer as internal space \cite{BJL15}, and certain model sets having a window with fat boundary in Euclidean internal space \cite{JLO18}.

The hull of squarefree integers is also of interest in number theory, as it is a factor of the M\"obius function flow. It can thus be used to understand elementary properties of the M\"obius function. This was made explicit in Sarnak's influential  article \cite{Sarnak2011}. The same idea applies to more general 
$\cB$-free systems \cite{BKKL2015}, which were analysed from a dynamical perspective using arguments of arithmetic nature. On the other hand, $\cB$-free systems are weak model sets with compact windows that may or may not have interior points, see subsection \ref{subsec:Bfree}. Thus one might suspect that results about $\cB$-free systems admit extensions to a suitable class of weak model sets and that, vice versa, a weak model set analysis could shed some additional light on $\cB$-free systems from a geometric perspective.

Let us mention two recent approaches along these lines. In \cite{BaakeHuck2015},  squarefree integer hull results from \cite{Sarnak2011} were formulated and proved for visible lattice points. This was extended to a larger class of weak model sets in \cite{BaakeHuckStrungaru15}. Given a condition called maximal density, pure point diffractivity was shown for such weak model sets. As maximal density is satisfied for model sets with windows having an almost vanishing boundary, this generalises previous results about pure point diffraction. This was then used to infer pure point dynamical spectrum for the hull equipped with the pattern frequency measure. The proof used approximation by regular model sets, a technique inspired by \cite{HuckRichard15}.

In \cite{KR2015}, the torus parametrisation was systematically re-investigated for weak model sets having compact windows. Results for the hull were deduced
from a larger dynamical system $\MWG$, which one may call the \textit{extended hull}. It is the translation orbit closure of the set of model sets arising from any window translate.
Whereas this approach was already used in \cite{Robinson2007}, properties of the extended hull were inferred from a related extended hull $\MW$, where model sets are viewed as lattice subsets in $G\times H$. Under transparent assumptions on the window, the above almost isomorphisms were quite easily deduced. Properties of $\MWG$ could then systematically be deduced from those of $\MW$ via a continuous factor map $\piG_*$, which describes the projection to $G$. If the above topological standard assumptions on the window are satisfied, then the map $\piG_*$ is indeed a homeomorphism \cite[Prop.~4.8]{KR2015}, so no information is lost. 
Measure-theoretic results were obtained for the Mirsky measure $\QMG$ on $\MWG$. This measure is the lift of the Haar measure $m_\hX$ on $\hX$ to $\MWG$ and agrees with the pattern frequency measure \cite[Rem.~3.12]{KR2015}. 
As the extended hull $(\MWG, \QMG)$ is a measure-theoretic factor of the torus, pure point dynamical spectrum follows immediately \cite[Thm.~2]{KR2015}\,\footnote{The corresponding factor map $\nuWG$ was already systematically used for weighted model sets in \cite{LenzRichard07}, where it is continuous.}. These results were then transferred to the usual hull of a model set. For measure-theoretic results the above maximal density condition turned out to be a sufficient condition ensuring isomorphism to the extended hull \cite[Thm.~5]{KR2015}.

In this article, the nature of $\piG_*$ is studied more systematically. This leads to a rather complete analysis of the maximal equicontinuous and Kronecker factors. For an initial discussion, let us restrict to compact windows. One of our results states that  $\piG_*$ is a homeomorphism whenever $\inn(W)$ is aperiodic. Hence for topological results the periods of $\inn(W)$ appear to play a central role. Indeed we obtain almost isomorphism to the maximal equicontinuous factor of $\MWG$ by factoring from the group $\hX$ the periods of $\inn(W)$. These results transfer easily to the usual hull, see Theorem~\ref{theo:inn-periodic} and also Theorem~\ref{theo:inn-periodicprime} for non-compact windows. If the window has no interior point, the maximal equicontinuous factor is trivial.
For measure-theoretic results, we first restrict to the Mirsky measure $\QMG$ on $\MWG$. Here a central role is played by the so-called Haar periods of $W$, i.e., by those $h\in H$ satisfying $m_H((h+W)\triangle W)=0$. Indeed the dynamical system $(\MWG,\QMG)$ is isomorphic to its Kronecker factor, which is obtained from the torus $\hX$ by factoring out the Haar periods of $W$, compare Theorem~\ref{theo:main-2-periodic}.  Given the maximal density condition, these results can be transferred to the usual hull. For any ergodic invariant probability measure $\PG$ which is not supported on the zero configuration, we
provide a kind of lower bound for the Kronecker factor of $(\MWG,\PG)$ in Theorem~\ref{theo:main-periodic}.

Let us interpret the above results in terms of diffraction of a weak model set of maximal density. 
The Bragg peaks of the weak model set generate a countable discrete group \cite[Thm.~9]{BaakeLenz2004}, whose dual is isomorphic to the Kronecker factor of its hull equipped with the Mirsky 
measure, see \cite[Sec.~7]{Lenz2009} for details. This Kronecker factor is determined in Theorem~\ref{theo:main-2-periodic}.
The amplitudes of the Bragg peaks appear as  squared norms of corresponding eigenfunctions \cite[Cor.~2]{Lenz2009}.
The maximal equicontinuous factor of the hull, which is determined in Theorem~\ref{theo:inn-aperiodic}, can be identified with the subgroup of eigenvalues which have continuous eigenfunctions. Apparently, the nature of the eigenfunctions is related to the different types of window periods.  The amplitude of a Bragg peak can be computed as a certain average over the corresponding eigenvalue \cite[Thm.~5(b)]{Lenz2009}. It would be interesting to analyse how continuity -- or weakened versions thereof as in \cite{Keller16} -- affect the type of convergence of these averages, compare \cite[Thm.~5(c)]{Lenz2009}.

Our paper is organised as follows. After the setting has been explained in Section~\ref{sec:setting}, we give formally precise statements of our results in Section~\ref{ch:results}. There we also discuss applications of the results to $\cB$-free systems. A more detailed discussion about non-trivial topological behaviour in $\cB$-free systems appears in  \cite{KKL2016}. Section~\ref{sec:proof-1} studies the question of reconstructing from a given weak model set a suitable window. This leads to a proof of Theorem~\ref{theo:inn-aperiodic}, assuming that $\inn(W)$ is aperiodic. Proofs of measure-theoretic statements for aperiodic windows are then provided in Section~\ref{sec:proof-2-3}. As a preparation for the proofs of the remaining statements, period groups and quotient cut-and-project schemes are studied in Section~\ref{sec:periods}. The following section contains the proofs of Theorems~\ref{theo:inn-periodic}, \ref{theo:main-2-periodic} and \ref{theo:main-periodic}. The final section discusses relatively compact windows whose associated dynamical systems behave very similarly to the ones with compact windows.

\section{The setting}\label{sec:setting}

The following point of view on extended weak model sets was developed in \cite{KR2015}.

\subsection{Assumptions and notations}\label{assnot}
Certain spaces and mappings are needed for the construction of weak model sets. 
As in \cite{KR2015} we make the
following general assumptions.
\begin{enumerate}[(1)]
\item $G$ and $H$ are \emph{locally compact second countable abelian groups} with Haar measures $m_G$ and $m_H$. Then the product group $G\times H$ is locally compact second countable abelian as well, and we choose $m_{G\times H}=m_G\times m_H$ as Haar measure on $G\times H$. 
\item $\LL\subseteq G\times H$ is a \emph{cocompact lattice}, i.e., a discrete subgroup whose quotient space $(G\times H)/\LL$ is compact. Thus $(G\times H)/\LL$ is a compact second countable abelian group.
Denote by $\piG:G\times H\to G$ and $\piH:G\times H\to H$ 
the canonical projections. We assume that 
$\piG|_\LL$ is \oneone and that
$\piH(\LL)$ is dense in $H$.\footnote{Denseness of $\piH(\LL)$ can be assumed without loss of generality by passing from $H$ to the closure of $\piH(\LL)$. In that case $m_H$ must be replaced by $m_{\,\overline{\piH(\LL)}}$.} 
\item $G$ acts on $G\times H$ by translation: $T_gx=(g,0)+x$.
\item Let $\hX:=(G\times H)/\LL$.  As we assumed that $\hX$ is compact, there is a measurable relatively compact fundamental domain $X\subseteq G\times H$ such that $x\mapsto x+\LL$ is a bijection between $X$ and $\hX$. Elements of $G\times H$ (and hence also of $X$) are denoted as $x=(x_G,x_H)$, elements of $\hX$ as 
$\hx$ or as $x+\LL=(x_G,x_H)+\LL$, when a representative $x$ of $\hx$ is to be stressed. We normalise the Haar measure $m_\hX$ on $\hX$ such that $m_\hX(\hX)=1$. Thus $m_\hX$ is a probability measure. 

\item The \emph{window} $W$ is a measurable relatively compact subset of $H$. 
For our topological dynamical results, we first assume that the window $W$ is indeed compact and discuss extensions of the results to certain non-compact windows in Section~\ref{sec:rcwin}. 
Our purely measure-theoretic results are first stated and proved for compact windows as well, but they extend easily to windows which agree modulo Haar measure zero with a compact one. Some further measure-theoretic results, which have an additional topological aspect, are only proved for compact windows.
\end{enumerate}

\subsection{Consequences of the assumptions}\label{en:ass}
We list a few facts from topology and measure theory that follow from the above assumptions.  We will call any neighborhood of the neutral element in an abelian topological group a \textit{zero neighborhood}.

\begin{enumerate}[(1)]
\item Being locally compact second countable abelian groups, $G$, $H$ and $G\times H$  are  metrisable with a translation invariant metric with respect to which they are complete metric spaces. In particular they have the Baire property.
As such groups are $\sigma$-compact, $m_G$, $m_H$ and $m_{G\times H}$ are $\sigma$-finite.
\item As $G\times H$ is $\sigma$-compact, the lattice $\LL\subseteq G\times H$ is at most countable.  Note that $G\times H$ can be partitioned by shifted copies of the relatively compact fundamental domain $X$. This means that $\LL$ has a positive finite point density $\dL=1/m_{G\times H}(X)$. We thus have $m_\hX(\hat A)=\dL\cdot m_{G\times H}(X\cap (\pihX)^{-1}(\hat A))$ for any measurable $\hat A\subseteq \hX$, where $\pihX: G\times H\to \hX$ denotes the quotient map. As a factor map between topological groups, $\pihX$ is continuous and open.
\item\label{item2.2:4} 
The action $\hT_g:\hx\mapsto(g,0)+\hx$ of $G$ on $\hX$ is minimal. 
This implies that $\hX$ with its natural action is uniquely ergodic, see e.g.~\cite[Prop.~1]{Moody2002}.

\item Denote by $\cM$ and $\cMG$ the spaces of all locally finite measures on the Borel subsets of $G\times H$ and $G$, respectively. They are endowed with the topology of vague convergence and hence compact spaces. 
As $G$ and $G\times H$ are complete metric spaces, this topology is Polish, see \cite[Thm.~A.2.3]{Kallenberg2001}. 
\end{enumerate}

\subsection{The objects of interest}
The pair $(\LL,W)$ assigns to any point $\hx\in \hX$ a discrete point set in $G\times H$. 
We will identify such point sets $P$ with the measure $\sum_{y\in P}\delta_y\,\in\cM$
and call these objects \emph{configurations}. More precisely:
\begin{enumerate}[(1)]
\item For $\hx=x+\LL\in\hX$ define the configuration
\begin{equation}\label{eq:nuW-def}
\nuW(\hx):=\sum_{y\in (x+\LL)\cap(G\times W)}\delta_y\ .
\end{equation}
It is important to understand $\nuW$ as a map from $\hX$ to $\cM$. If $W$ is compact, the map $\nuW$ is upper semicontinuous \cite[Prop.~3.3]{KR2015}.
The canonical projection $\piG:G\times H\to G$ projects measures $\nu\in\cM$ to measures $\piG_*\nu$ on $G$ defined by $\piG_*\nu(A):=\nu((\piG)^{-1}(A))$. We abbreviate 
\begin{equation}\label{eq:nuWG-def}
\nuWG:=\piG_*\circ\nuW:\hX\to\cMG \ .
\end{equation}
\item\label{item:spaces}
 Denote by $\MW$ the vague closure of $\nuW(\hX)$ in $\cM$, and by $\MWG$ the vague closure of $\nuWG(\hX)$ in $\cMG$.
The group $G$ acts continuously by translations on all these spaces:
$(S_g\nu)(A):=\nu(T_g^{-1}A)=\nu(T_{-g}A)$.
Here we used the same notation  $S_g$ for translations on $\MW$ and $\MWG$, as the meaning will always be clear from the context. 
\item As $\nuW(\hx)(T_{-g}A)=(S_g\nuW(\hx))(A)=\nuW(\hT_g\hx)(A)$, it is obvious that all $\nuW(\hx)$ are uniformly translation bounded, and it follows from \cite[Thm.~2]{BaakeLenz2004} that all four spaces from item (\ref{item:spaces}) are compact.
\item $\QM:=m_\hX\circ {\nuW}^{-1}$ and 
$\QMG:=m_\hX\circ (\nuWG)^{-1}$ are called Mirsky measures on 
$\MW$ and $\MWG$, respectively. Note that $\QMG=\QM\circ (\piG_*)^{-1}$.
\footnote{These measures were denoted $Q_{\cM}$ resp. $Q_{\cM^G}$ in \cite{KR2015}.}
\end{enumerate}

\subsection{Previous results}

For compact windows, Mirsky measures on $\MW$ or on $\MWG$ were studied in quite some detail in \cite{KR2015}. The following property is immediate from measurability of the map $\nuW:\hX\to\MW$ and from the definition of the Mirsky measure $\QM$ on $\MW$.

\medskip

\begin{proposition}\label{prop:MWfactor}
$(\MW,\QM,S)$ is a measure-theoretic factor of $(\hX,m_\hX,\hT)$. \qed
\end{proposition}
In this article, we aim at statements concerning general invariant probability measures on $\MW$ or on $\MWG$. This is achieved using a partial inverse of $\nuW$: Denote by $\0\in\cM$  the zero measure (``empty configuration''). We have $\0\in\MW$ if and only if $\inn(W)=\emptyset$ by \cite[Prop.~3.3]{KR2015}.
Recall from \cite[Lem.~5.4]{KR2015} that, for each $\nu\in\MW\setminus\{\0\}$, there is a unique
$\hat\pi(\nu)\in\hX$, its ``torus parameter'', such that $\supp(\nu)\subseteq\supp(\nuW(\hat\pi(\nu)))$.
This yields a continuous map $\hat\pi: \MW\setminus\{\0\}\to\hX$, and we have $\hat\pi \circ\nuW=id_\hX$ whenever this composition is well defined,  compare \cite[Lem.~5.6]{KR2015}.

\medskip

 The following observation is a measure-theoretic analogue to Theorem 1a in \cite{KR2015}. Its proof, which is already implicit in the proof of \cite[Thm.~2a]{KR2015}, will be given in Subsection~\ref{subsec:mtr}.
\begin{proposition}\label{prop:factor}
Assume that $P$ is any $S$-invariant probability measure on $\MW$ satisfying $P(\MW\setminus \{\0\})=1$. Then $m_H(W)>0$, and  $(\MW,P,S)$ is a measure-theoretic extension of $(\hX,m_\hX,\hT)$. 
\end{proposition}

Specialising to the Mirsky measure, we can combine the above two propositions and recover the following result. For the convenience of the reader, its proof will be given in Subsection~\ref{subsec:mtr}.

\begin{proposition}[Theorem 2a in \cite{KR2015}]\label{prop:mirsky}
Assume that $m_H(W)>0$. Then $(\MW,\QM,S)$ is measure-theoretically isomorphic to $(\hX,m_\hX,\hT)$.
\end{proposition}

The projection $\piG:G\times H\to G$ induces a continuous factor map
$\piG_*:(\MW,S)\to(\MWG,S)$, which is the object of interest in this article. 
In order to understand to which extent statements as in the above propositions carry over to the system $(\MWG,S)$, one has to understand the degree of (non)invertibility of $\piG_*$.

Recall that a window  $W\subseteq H$ is \textit{(topologically) aperiodic}  or \textit{irredundant}, if   $h+W=W$ implies $h=0$.  In particular, any aperiodic window is nonempty. A window  $W\subseteq H$ is called \emph{topologically regular} if $W=\overline{\inn(W)}$.  Note that if $W\neq\emptyset$ is topologically regular, then $m_H(W)>0$.
The following fact is the essence of the results in \cite[Sec.~4.3]{KR2015}:

\begin{fact}\label{coro:top-regular-aperiodic}
$\piG_*:\MW\to\MWG$ is a homeomorphism, whenever the window $W$ is aperiodic and topologically regular.
\end{fact}

\noindent This fact is substantially extended in Theorems~\ref{theo:inn-aperiodic},~\ref{theo:inn-periodic},~\ref{theo:inn-aperiodicprime} and~\ref{theo:inn-periodicprime}.

\medskip

We now turn to measures on $\MWG$.
For the Mirsky measure $\QMG$ on $\MWG$,
Proposition~\ref{prop:MWfactor} immediately implies that  $(\MWG,\QMG,S)$ is a measure-theoretic factor of $(\hX,m_\hX,\hT)$. This implies pure point dynamical spectrum for the former system and, together with a condition called maximal density,  pure point dynamical spectrum transfers to the usual hull \cite[Cor.~6]{KR2015}, compare Remark~\ref{rem:hulls} and the result  \cite[Cor.~17]{BaakeHuckStrungaru15}. But there is no statement analogous to Proposition~\ref{prop:factor}. 
From Fact~\ref{coro:top-regular-aperiodic} and Proposition~\ref{prop:mirsky}, however, we get the following result for the Mirsky measure $\QMG$ on $\MWG$.
\begin{fact}\label{cor:miriso}
Suppose  that $W$ is aperiodic and topologically regular.
Then  $(\MWG,\QMG,S)$ is measure-theoretically isomorphic to $(\hX,m_\hX,\hT)$. \qed
\end{fact}

\noindent Theorems~\ref{theo:main-2} and~\ref{theo:main-2-periodic} extend this fact.
The following is an immediate consequence of Fact~\ref{coro:top-regular-aperiodic} and Proposition~\ref{prop:factor}. It will be extended in Theorems~\ref{theo:main} and~\ref{theo:main-periodic}.
\begin{fact}\label{cor:mtext}
Suppose  that $W$ is aperiodic and topologically regular.
If $\PG$ is an $S$-invariant probability measure on $\MWG$, then  $(\MWG,\PG,S)$ is a measure-theoretic extension of $(\hX,m_\hX,\hT)$. \qed
\end{fact}

\section{Main results}\label{ch:results}

\subsection{Topological results}
In this subsection we assume that the window $W$ is compact. Extensions of the results to certain non-compact windows are discussed in Section~\ref{sec:rcwin}.
Our first main result strengthens Fact~\ref{coro:top-regular-aperiodic} considerably.

\begin{theoremI}\label{theo:inn-aperiodic}
Assume that $W$ is compact and that $\inn(W)$ is aperiodic (so in particular non-empty). 
\begin{compactenum}[a)]
\item The topological dynamical systems $(\MW,S)$ and $(\MWG,S)$ are isomorphic, and both are almost \oneone extensions of their maximal equicontinuous factor $(\hX,\hT)$.
\item Denote by $\Gamma:\MWG\to\hX$ the factor map from a).
If $M$ is a non-empty, closed $S$-invariant subset of $\MWG$, then $(M,S)$ is an
almost \oneone extension of its maximal equicontinuous factor $(\hX,\hT)$ with factor map 
$\Gamma|_{M}$.
\end{compactenum}
\end{theoremI}

\begin{remark}\label{remark:aperiodic}
\begin{compactenum}[a)]
\item The above result extends \cite[Cor.~1]{KR2015}, as aperiodicity of $\overline{\inn(W)}$ implies aperiodicity of $\inn(W)$ by Lemma~\ref{lem:aper}a).  
In particular, the result applies to $M=\MminG$ the unique minimal subset of $\MWG$, see Remark~\ref{remark:MminG}, and to $M=\MWG(\hx):=\overline{\{S_g\nuWG(\hx):g\in G\}}$ the so-called \textit{hull} of $\nuWG(\hx)\in \MWG$.

\item
Recall that $\inn(W)\ne\emptyset$ is aperiodic whenever the only  compact subgroup of $H$ is the trivial one, compare the proof of \cite[Prop.~4.8]{KR2015}. This holds in particular for $H=\R^d$.

\item Aperiodicity of $\inn(W)$ implies aperiodicity of $W$. Examples of aperiodic $W$ with periodic interior are presented in \cite{KKL2016}.

\item For a topologically regular window, aperiodicity of $W$ and $\inn(W)$ are equivalent by Lemma~\ref{lem:aper}a). This gives the setting of an earlier result of Robinson \cite[Th.~5.19, Cor.~5.20]{Robinson2007}, see also \cite[Prop.~7.3]{FHK2002} in the Euclidean situation. 

\end{compactenum}

\end{remark}

If $\inn(W)$ has non-trivial periods, we still can determine the maximal equicontinuous factor of $(\MWG,S)$. Given a subset $A\subseteq H$, we call
\begin{equation*}
H_A:=\left\{h\in H: h+A=A\right\}
\end{equation*}
the \emph{period group} of $A$. The set $A\subseteq H$ is \textit{(topologically) aperiodic}, if $H_A=\{0\}$.
The following result extends Theorem~\ref{theo:inn-aperiodic}.  We denote by $\cH_A:=\{0\}\times H_A\subseteq G\times H$ the canonical embedding of $H_A$ into $G\times H$. Note that $\cH_{\inn(W)}$ is compact as $H_{\inn(W)}=H_{\overline{\inn(W)}}$ is a compact subgroup of $H$, compare  Lemma~\ref{lem:aper}c).

\begin{theoremII}\label{theo:inn-periodic}
Assume that $W$ is compact and that $\inn(W)\neq\emptyset$.
Let $\hXprime=\hX/\pihX(\cH_{\inn(W)})$ with induced $G$-action $\hTprime$, 
and let $M$ be any non-empty, closed $S$-invariant subset of $\MWG$ (thus including the case $M=\MWG$).
\begin{compactenum}[a)]
\item $(\hXprime,\hTprime)$ is the maximal equicontinuous factor of
the topological dynamical system $(M,S)$.
\item If $H_{\inn(W)}=H_W$, then $(M,S)$ is an almost \oneone extension of $(\hXprime,\hTprime)$.
\end{compactenum}
\end{theoremII}

\begin{remark}
\begin{compactenum}[a)]
\item We do not know whether $H_{\inn(W)}=H_W$ is also a necessary condition in part b) of the theorem. The condition is satisfied for topologically regular windows.

\item 
If $\inn(W)=\emptyset$, then the maximal equicontinuous factor of $(M,S)$ is trivial, 
see Remark~\ref{remark:MminG}. However statements analogous to the above two theorems hold for maximal equicontinuous generic factors \cite{Keller16}.
\end{compactenum}
\end{remark}

\subsection{Measure-theoretic results}

In this subsection we assume that $W\subseteq H$ is relatively compact and measurable, but not necessarily compact.
If $\inn(W)=\emptyset$, one should not expect that $\piG_*$ is a homeomorphism.
However, the measure-theoretic statements of Facts~\ref{cor:miriso} and~\ref{cor:mtext} can still be generalized substantially to windows where  $\piG_*:\MW\to\MWG$ is not necessarily a homeomorphism, but still \oneone on a sufficiently large subset of $\MW$. This is achieved by replacing topological aperiodicity through a stronger measure-theoretic version. In the following definition, $\triangle$ denotes the symmetric set difference.

\begin{definition}[Haar aperiodicity]
A measurable set $A\subseteq H$ is 
\emph{Haar aperiodic}, if $m_H((h+A)\triangle A)=0$ implies $h=0$.
\end{definition}

\begin{remark}
Any Haar aperiodic set $A$ satisfies $m_H(A)>0$ and is, in particular, nonempty. Haar aperiodicity implies topological aperiodicity, but the converse holds only for a more restricted class of windows, see Remark~\ref{remark:Haar-and-other-periods}.
\end{remark}

\begin{definition}
A measurable subset $A$ of $H$ is \emph{compact modulo $0$}, if there is a compact
set $K\subseteq H$ such that $m_H(A\triangle K)=0$.
\end{definition}

We have the following results for the Mirsky measure $\QMG$  on $\MWG$ , which generalise Fact~\ref{cor:miriso}.  They are proved in Section~\ref{sec:proof-2-3}. 

\begin{theoremI}\label{theo:main-2}
Suppose that  $W$ is compact modulo $0$ and Haar aperiodic. Then $(\MWG,\QMG,S)$ is measure-theoretically isomorphic to $(\hX,m_\hX, \hT)$.
\end{theoremI}

Here is an extension of Theorem~\ref{theo:main-2} to windows that are not Haar aperiodic. 
The proof is provided in Section~\ref{sec:proofs-II} after some preparations in Section~\ref{sec:periods}.
To formulate the statement, we consider the group $H_W^{Haar}$ of Haar periods of $W$, i.e.,
\begin{displaymath}
H_W^{Haar}:=\{h\in H: m_H((h+W)\triangle W)=0\} \ .
\end{displaymath}
We write $\cH_W^{Haar}=\{0\}\times H_W^{Haar}$ for the canonical embedding of $H_W^{Haar}$ into $G\times H$.

\begin{theoremII}\label{theo:main-2-periodic}
Suppose that $W$ is compact modulo $0$ and $m_H(W)>0$. Let $\hXprime=\hX/\pihX(\cH_W^{Haar})$ with induced $G$-action $\hTprime$ and Haar measure $m_{\hXprime}$. Then the measure-theoretic dynamical system $(\MWG,\QMG, S)$ is isomorphic to $(\hXprime,m_\hXprime,\hTprime)$. 
\end{theoremII}

\begin{remark}\label{rem:hulls}
In order to transfer these results to hulls $\MWG(\hx)\subseteq \MWG$, we may assume that $\nuWG(\hx)$ has maximal density,  compare Remark~\ref{rem:genconf}. 
In that case $\QMG(\MWG(\hx))=1$, see \cite[Cor.~5]{KR2015},
and the two systems $(\MWG(\hx),\QMG,S)$ and
 $(\MWG,\QMG,S)$ are measure-theoretically isomorphic.
\end{remark}

The two final results refer to general invariant measures on $\MWG$ and generalise Fact~\ref{cor:mtext}. 
They are of \emph{measure-theoretic} nature, but they provide information about 
arbitrary ergodic measures on $\MWG$ - the \emph{topological} closure of $\nuW(\hX)$.
Therefore they are stated and proved only for compact windows. Hence in the remainder of this subsection we assume that $W\subseteq H$ is compact.

Before we can state the results, we need to introduce one more concept. For each $\nu\in\MW$, $\piH_*\nu$ is a measure%
\footnote{ 
Note that generally $\piH_*\nu$ is not a Borel measure as $\nu$ may be an unbounded configuration and the topological support of $\piH_*\nu$ lies inside the compact set $W$.} 
on $H$, and we denote the topological support of this measure by $\SH(\nu)$. We thus have
\begin{equation}\label{eq:SH(nu)}
\SH(\nu)=\supp(\piH_*(\nu))=\overline{\piH(\supp(\nu))}\subseteq W\ .
\end{equation}

\begin{remark}\label{rem:SHcont}
\begin{compactenum}[a)]
\item Clearly $\SH(\0)=\emptyset$. Recall that there is a unique $\hx\in\hX$ such that $\nu\le \nuWbar(\hx)$ if $\nu\ne\0$. The set $\SH(\nu)$ is the smallest compact set $W'\subseteq W$ such that $\nu\le \nuWprime(\hx)$.
\item We have $\overline{\inn(W)}\subseteq \SH(\nu)\subseteq W$
for any $\nu\in\MW$ by Lemma~\ref{lemma:SH-inv}. The lower bound is attained for any continuity point $\hx=x+\LL$ of the map $\nuW$, because  $\SH(\nuW(\hx))=\overline{\inn(W)\cap \piH(x+\LL)}=\overline{\inn(W)}$
for such $\hx$ by Eqn.~\eqref{eq:CW-def} below.
\item In Lemma~\ref{lemma:SH-inv-ergod} we prove:
For each ergodic $S$-invariant probability measure $P$ on $\MW$, there
is a compact subset $\WP\subseteq W$ of $H$ such that 
$\SH(\nu)=\WP$ for $P$-a.a. $\nu$. 
It should be no surprise that 
$m_H(W\triangle W_{\QM})=0$ when $W$ is compact modulo $0$,
see Corollary~\ref{coro:WQM=Wreg}.
\end{compactenum}
\end{remark}

The first result is a consequence of Proposition~\ref{prop:main} below.
\begin{theoremI}\label{theo:main}
Suppose that $W$ is compact and $m_H(W)>0$.
Let $\PG$ be an ergodic  $S$-invariant probability measure on $\MWG$\,, and let $P$ be any ergodic $S$-invariant probability measure on $\MW$ satisfying $\PG=P\circ(\piG_*)^{-1}$. (There always exists at least one such measure $P$, see Proposition~\ref{prop:main}.)
Suppose that $\WP$ is aperiodic. (In particular $W_P\neq\emptyset$, i.e. $P(\MW\setminus\{\0\})=\PG(\MWG\setminus\{\0\})=1$.)
Then $\piG_*$ is a measure-theoretic isomorphism between $(\MW,P,S)$ and $(\MWG,\PG,S)$, and $(\MWG,\PG,S)$  is a measure-theoretic extension of $(\hX,m_\hX, \hT)$.
\end{theoremI}

\begin{remark}\label{rem:main}
\begin{compactenum}[a)]
\item
As $W_{\QM}$ is aperiodic iff $W$ is Haar aperiodic by Remark~\ref{remark:Haar-and-other-periods} and Corollary~\ref{coro:WQM=Wreg}, the above statement is consistent with Theorem~\ref{theo:main-2}. Aperiodicity of $W_P$ holds in $\R^d$, see Remark~\ref{remark:aperiodic}.
\item
The above result does not depend on the choice of $P$. 
Indeed we have $H_{W_P}=H_{W_{P'}}$, whenever $P$ and $P'$ are ergodic $S$-invariant measures on $\MW$ with 
$P\circ(\piG_*)^{-1}=P'\circ(\piG_*)^{-1}$. This holds
as there is $d\in H$ such that  $W_P=W_{P'}+d$ by Lemma~\ref{lemma:SH-inv-ergod}. Hence,
if $h\in H_{W_{P'}}$, then $h+W_P=h+(W_{P'}+d)=(h+W_{P'})+d=W_{P'}+d=W_P$, so that $H_{W_{P'}}\subseteq H_{W_P}$, and the reverse inclusion follows from interchanging the roles of $P$ and $P'$.
\end{compactenum}
\end{remark}

Again there is a periodic generalisation of this theorem, which is proved in Section~\ref{sec:proofs-II}.

\begin{theoremII}\label{theo:main-periodic}
Suppose that $W$ is compact and  $m_H(W)>0$.
Let $\PG$ be an ergodic  $S$-invariant probability measure on $\MWG\setminus\{\0\}$. Take any ergodic $S$-invariant probability measure $P$ on $\MW$ satisfying $\PG=P\circ(\piG_*)^{-1}$. (There always exists at least one such measure $P$, see Proposition~\ref{prop:main}.)
Let $\hXprime=\hX/\pihX(\cH_{W_{P}})$ with induced $G$-action $\hTprime$ and Haar measure $m_{\hXprime}$. Then  $(\MWG,\PG, S)$ is a measure-theoretic extension of $(\hXprime,m_\hXprime,\hTprime)$. 
\end{theoremII}

\begin{remark}
The above result does not depend on the choice of $P$, see Remark~\ref{rem:main}b).
\end{remark}

For some of the proofs and applications, the following notion of Haar regularity appears to be relevant, compare Lemma~\ref{lemma:SH-inv-ergod}.  It is a measure-theoretic substitute for topological regularity.

\begin{definition}[Haar regularity]\label{def:Haar-regular}
Consider some compact subset $K$ of $H$.
\begin{compactenum}[a)]
\item The compact set $K_{reg}$, defined as the topological support of the measure $(m_H)|_K$, is called the \emph{Haar regularization} of $K$.
\item The set $K$ is \emph{Haar regular}, if
$K=K_{reg}$. 
\end{compactenum}
\end{definition}

\begin{remark}\label{rem:hr}
\begin{compactenum}[a)]
\item
A compact set $K\subseteq H$ is Haar regular if and only if  for every open $U\subseteq H$ such that $U\cap K\ne \emptyset$ we have $m_H(U\cap K)>0$.  
\item $\overline{\inn(K)}\subseteq K_{reg}\subseteq K$ and $m_H(K\setminus K_{reg})=0$ by definition.
\item
The empty set is Haar regular. If $K\ne\emptyset$ is Haar regular, then $m_H(K)>0$.
\item If $K$ is Haar regular, then any translate of $K$ is Haar regular.  If $K$ is topologically regular, then $K$ is Haar regular.
\end{compactenum}
\end{remark}

\begin{remark}(Periods and Haar periods)\label{remark:Haar-and-other-periods}\\
A compact set $W$ is Haar aperiodic if and only if $W_{reg}$ is aperiodic.
More generally, $H_W^{Haar}=H_{W_{reg}}$. 
Indeed, as $W_{reg}\subseteq W$, we have $h+W_{reg}=W_{reg}$ if and only if $m_H((h+W)\triangle W)=0$, which implies the claimed equivalence.
In particular, every aperiodic Haar regular window is
Haar aperiodic. 
Examples are Haar regular windows in $\R^d$ 
or the windows defining taut $\cB$-free systems, see Subsection~\ref{subsec:Bfree}. 
\end{remark}

\begin{remark}(Haar regularity and compactness modulo $0$)\label{remark:reg-and-per}\\
If $A\subseteq H$ is compact modulo $0$, then there is a compact set $K \subseteq H$ such that $m_H(A\triangle K)=0$. Hence also $m_H(A\triangle K_{reg})=0$, and $K_{reg}$ is the unique Haar regular set with this property. Therefore we denote it by $A_{reg}$. 
Observe that $A_{reg}=\supp((m_H)|_{A})$.
\end{remark}

\subsection{Applications to $\cB$-free dynamics}
\label{subsec:Bfree}

General $\cB$-free dynamical systems were studied in \cite{BKKL2015}. They are a special case of our systems $(\MWG,S)$, when $G=\Z$ and $H$ is a particular compact group constructed from the given set $\cB\subseteq\N$. 
In this setting a configuration $\nuG=\sum_{n\in A}\delta_n\in\cMG$, where $A$ is a subset of $G=\Z$, can be identified with the characteristic function $\chi_A$ interpreted as an element of $\{0,1\}^\Z$.
Our Theorem~\ref{theo:main-2} reproduces Theorem F of \cite{BKKL2015} in this context,
i.e.~the measure-theoretic dynamical systems $(\MWG,\QMG,S)$ and $(\hX,m_{\hX},\hT)$ are isomorphic.

Since the connection with \cite[Theorem F]{BKKL2015} is not completely obvious, we give some explanation: 
We assume without loss of generality that the set $\cB$ is \emph{primitive}, i.e. that no number from $\cB$ is a multiple of another number from $\cB$.
The following group homomorphism is associated with the set $\cB$:
\begin{equation*}
\Delta_\cB:\Z\to\prod_{b\in\cB}\Z/b\Z\ ,\quad\Delta_\cB(n)=(n\text{ mod }b)_{b\in\cB}\ ,
\end{equation*}
and $H$ is the topological closure of $\Delta_\cB(\Z)$. The lattice is $\LL=\left\{(n,\Delta_\cB(n))\in\Z\times H: n\in\Z\right\}$, and
a moment's reflection shows that the group $\hX=(\Z\times H)/\LL$ is isomorphic to $H$.
The window is the compact set defined as
\begin{equation}\label{eq:B-free-W}
W=\left\{h\in H: h_b\neq0\; \forall b\in\cB\right\}.
\end{equation}
With this notation, an integer $n$ is $\cB$-free, i.e.~is not divisible by any number $b\in\cB$, if and only if $\Delta_\cB(n)\in W$. Hence $\Delta_\cB^{-1}(W)$ is precisely the set of $\cB$-free integers, and $\nuWG(0)$ coincides with the respective configuration $\sum_{n\in\Delta_\cB^{-1}(W)}\delta_n$. The window $W$ is topologically regular if and only if the set $\cB$ contains no scaled copy of an infinite pairwise coprime set \cite[Theorem~B]{KKL2016}. A most prominent example of this type is the set $\cB$ consisting of the squares of all primes, in which case $\Delta_\cB^{-1}(W)$ is the set of all squarefree integers.

The authors of \cite{BKKL2015} study, in our notation, the dynamics of $S$ restricted to $\overline{\nuWG(\Delta_\cB(\Z))}$, and denote this set by $X_\eta$. Our $\nuWG:\hX\to\MWG$ corresponds to their $\varphi:H\to\{0,1\}^\Z$ \cite[Section 2.7]{BKKL2015}\footnote{Observe that the compact group $H$ is denoted by $G$ in \cite{BKKL2015}.}, and so their definition of a Mirsky measure $\nu_\eta$ \cite[Section 2.9]{BKKL2015} translates precisely to our $\QMG$. 

Finally observe that the tautness
assumption on $\cB$ in \cite[Thm.~F]{BKKL2015} is equivalent to the Haar regularity of the window  $W$ defined in Eqn.~\eqref{eq:B-free-W}, see \cite[Thm.~A]{KKL2016}, and that $W$ is Haar aperiodic by Remark~\ref{remark:Haar-and-other-periods}, because it is Haar regular and aperiodic \cite[Prop.~5.1]{KKL2016}.

In view of the preceding discussion, also our Theorem~\ref{theo:main} applies to $\cB$-free systems. It complements Theorem~I from \cite{BKKL2015}, which we recall here using our notation:\footnote{To be precise, \cite[Thm.~I]{BKKL2015} is stated for general, not necessarily ergodic measures $\PG$, and it applies also to measures not supported by $X_\eta$, but by its hereditary closure.}
\begin{quote}
For any ergodic $S$-invariant probability measure $\PG$ on $X_\eta$, there exists an $S$-invariant probability measure $\rho$ on $X_\eta\times\{0,1\}^\Z$ whose first marginal is $\QMG$ and such that $\rho\circ M^{-1}=\PG$, where $M:X_\eta\times\{0,1\}^\Z\to \{0,1\}^\Z$ stands for the coordinatewise multiplication.
\end{quote}
Together with our Theorem~\ref{theo:main}, which  adds the lower arrow, this yields the following commutative diagram for measures $\PG$ with an aperiodic associated window $W_P$:
\begin{equation}\label{eq:diagram}
\begin{tikzcd}
\ &(X_\eta\times\{0,1\}^{\Z},\rho,S\times S)\arrow[swap]{ld}{\pi_{X_\eta}}\arrow{rd}{M}&\\
(X_\eta,\QMG,S)&\ &(X_\eta,\PG,S)\arrow[swap]{ll}
\end{tikzcd}
\end{equation}

In \cite[Thm.~8.2]{BKKL2015} the authors prove that the system
$(X_\eta,S)$ has a unique invariant measure $\PG_{max}=\rho_{max}\circ M^{-1}$ of maximal entropy, whenever the set $\cB$ has light tails\footnote{This means that for all $\epsilon>0$ there is $k_0\in\N$ such that for all $k\geqslant k_0$ the asymptotic upper density of the set $\bigcup_{b\in\cB,b>k}b\Z$ does not exceed $\epsilon$. It is known that the asymptotic density $d(\cF_\cB)=\lim_{N\to\infty}N^{-1}\card\left(\cF_\cB\cap\{1,\dots,N\}\right)$ exists when $\cB$ has light tails,
and hence coincides with the so called logarithmic density of this set \cite[Remark 2.29]{BKKL2015}.} and contains an infinite pairwise coprime subset. 
In fact, $\rho_{max}=\QMG\times B(1/2,1/2)$, where $B(1/2,1/2)$ denotes the Bernoulli\,-\,$(1/2,1/2)$ measure on $\{0,1\}^\Z$.
Here we complement their result by noticing that diagram \eqref{eq:diagram} applies to $\PG_{max}$:
\begin{proposition}
If $\cB$ has light tails and contains an infinite pairwise coprime subset,\footnote{\emph{Note added in proof:} These assumptions are only used to show that the subshift $X_\eta$ is hereditary, i.e.~that $M(X_\eta\times\{0,1\}^\Z)\subseteq X_\eta$ \cite[Theorem D]{BKKL2015}. More recently \cite[Theorem~3]{Keller2018}, the same conclusion was reached when the assumption ,,$\cB$ has light tails'' is replaced by the weaker ,,$\cB$ is taut'', which is equivalent to Haar regularity of the window $W$ defined in~\eqref{eq:B-free-W}.} then 
the dynamical system $(X_\eta,\PG_{max},S)$ is a measure-theoretic extension of $(\hX,m_\hX,\hT)$.
\end{proposition}
\begin{proof}
We must show that the window $W_{P_{max}}$ is aperiodic:
The entropy $h(\PG_{max})$ coincides with the topological entropy of the system, which in turn equals the asymptotic density $d(\cF_\cB)$ of the set $\cF_\cB$ of $\cB$-free numbers,
when the logarithm to base $2$ is used to compute the entropy,
see \cite[Thm.~D and Prop.~K]{BKKL2015}. Remark 4.2 of \cite{BKKL2015} then shows that
\begin{equation*}
\begin{split}
h(\PG_{max})
&=
d(\cF_\cB)
=
\QMG\left\{\nu\in\MWG: \nu(\{0\})>0\right\}
=
m_\hX\left\{\hx\in\hX: \nuWG(\hx)(\{0\})>0\right\}\\
&=
m_H\left\{h\in H: \varphi(h)(0)=1\right\}
=m_H(W)\ .	
\end{split}
\end{equation*}

Fix any ergodic invariant measure $P_{max}$ on $\MW$ that projects to $\PG_{max}$ (compare Proposition~\ref{prop:main}a). This measure
is supported by the set $\widetilde{\MWprime}:=\{\nu\in\MW: \exists\hx\in\hX\text{ s.t. }\nu\leqslant\nuWprime(\hx)\}$, where the window $W'$ is defined as $W'=W_{P_{max}}\subseteq W$, and so $\PG_{max}$ is supported by $\widetilde{\MWGprime}:=\{\nuG\in\MWG: \exists\hx\in\hX\text{ s.t. }\nuG\leqslant\nuWGprime(\hx)\}$. 

Let $\epsilon>0$. There is $N\in\N$ such that $\card\{k\in\N: |k|\leqslant n, \nuG\{k\}=1\}\leqslant n\cdot(m_H(W')+\epsilon)$ for all $n\geqslant N$ and all $\nuG\in\nuWGprime(\hX)$, see 
\cite[Thm.~1]{Moody00} or \cite[Thm.~3]{KR2015}. Hence the same holds for all $\nuG\in\widetilde{\MWGprime}$. Therefore the topological entropy of the subshift $\widetilde{\MWGprime}\subseteq\{0,1\}^\Z$
is at most $m_H(W')+\epsilon$. It follows that 
$m_H(W)=h(\PG_{max})\leqslant  m_H(W')$. (For related arguments see also \cite{HuckRichard15}.) Hence $m_H(W\setminus W')=0$.
Because sets $\cB$ with light tails are taut \cite[Section 2.6]{BKKL2015} and give rise to Haar regular windows \cite[Thm.~A]{KKL2016}, this implies $W_{P_{max}}=W'=W$. As, in the context of $\cB$-free systems, the window $W$ is always aperiodic \cite{KKL2016}, we see that $W_{P_{max}}$ is aperiodic.
\end{proof}

\section{The map $\SH$ and the proof of Theorem~\ref{theo:inn-aperiodic}}
\label{sec:proof-1}
Throughout this section we assume that $W$ is compact.
Recall from (\ref{eq:SH(nu)}) that, for each $\nu\in\MW$, $\piH_*\nu$ is a measure
on $H$, whose topological support is denoted by $\SH(\nu)$. The set $\SH(\nu)\subseteq W$ can be understood as the ``minimal'' window for $\nu$ in the following sense:
Assume that $\nu\in\MW$ satisfies $\nu\le\nuW(\hx)$ for some $\hx\in \hX$. Then the smallest compact set $W'\subseteq W$ such that $\nu\le\nuWprime(\hx)$ is given by $W'=\SH(\nu)$.

Denote $\widetilde{\MW}:=\{\nu\in\cM: \nu\leqslant\nuW(\hx)\text{ for some }\hx\in\hX\}$.
Then $\MW=\overline{\nuW(\hX)}\subseteq\widetilde{\MW} $, because $\nuW$ is upper semi-continuous.
It is advantageous to view $\SH$ as a map from $\widetilde{\MW}\setminus\{\0\}$ to $\KW$, the space of all non-empty compact subsets of~$W$, which is equipped with the topology generated by the Hausdorff distance.

\begin{lemma}\label{lem:lsc}
$\SH:\widetilde{\MW}\setminus\{\0\}\to\KW$ is lower semicontinuous, i.e.,
for each closed $F\subseteq W$ the \emph{core} of $F$, i.e. the set $\{\nu\in\widetilde{\MW}\setminus\{\0\}:\ \SH(\nu)\subseteq F\}$ is closed.
In particular, $\SH$ is Borel measurable. The same holds for the restriction
$\SH|_{\MW\setminus\{\0\}}$.
\end{lemma}

\begin{proof}
The above characterisation of lower semicontinuity is from \cite[Prop.~1.4.4]{AF1990}.
So let $\nu=\lim_{n\to\infty}\nu_n$ with $\SH(\nu_n)\subseteq F$. Suppose for a contradiction that $\SH(\nu)$ is not contained in $F$.
Then, by closedness of $F$, 
it follows that there is $h\in\SH(\nu)\setminus F$ such that $(\piH_*\nu)\{h\}=1$. Hence
there are $x\in G\times H$ and $\ell\in\LL$ such that $(x+\ell)_H=h$ and
$\nu\{x+\ell\}=1$. As $\nu_n\to\nu$ vaguely, there are $x_n\in G\times H$ such that $x_n\to x$ and $\nu_n\{x_n+\ell\}=1$ for all $n$. But then 
$(x_n+\ell)_H\in\SH(\nu_n)\subseteq F$ for all $n$, and
$(x_n+\ell)_H\to h$, so that $h\in F$, a contradiction.
This proves the lower semicontinuity of $\SH$, and its Borel measurability follows from
 \cite[Cor.~III.3]{CV1977}. 
 As $\MW$ is a closed subset of $\widetilde{\MW}$, these properties are inherited by the restriction $\SH|_{\MW\setminus\{\0\}}$.
\end{proof}
Denote by $\CW\subseteq\hX$ the set of continuity points of the map $\nuW:\hX\to\MW$. 
It is a dense $G_\delta$-set, see Proposition~\ref{lem:dGd} below.
An explicit characterization of this set is
\begin{equation}\label{eq:CW-def}
\CW=\pihX\left(\bigcap_{\ell\in\LL}\left(G\times(\partial W)^c)-\ell\right)\right) \ ,
\end{equation}
see Lemma~\ref{lem:charcw} below or e.g. \cite[Lem.~6.1]{KR2015}.

\begin{remark}\label{remark:MminG}
$\Mmin=\overline{\nuW(\CW)}$ is the unique minimal subset of $\MW$ by Lemma~\ref{lemma:minimal-subset} below, see also \cite[Thm.~1a]{KR2015}. 
Let $\MminG:=\piG_*(\Mmin)\subseteq\MWG$. Then $\MminG$ is the only minimal subsystem of $\MWG$. Even more, $\MminG=\overline{\nuWG(\CW)}$. The $\subseteq$-inclusion follows, because
$\overline{\nuWG(\CW)}$ is non-empty, closed and $S$-invariant. For the reverse inclusion observe that $\nuWG(\CW)=\piG_*(\nuW(\CW))\subseteq\piG_*(\Mmin)$.
If $\inn(W)=\emptyset$, then $\Mmin=\{\0\}$ is a singleton which consists only of the zero measure, see Lemma~\ref{lemma:minimal-subset} or \cite[Prop.~3.3]{KR2015}.
\end{remark}

\begin{lemma}\label{lemma:SH-inv}
\begin{compactenum}[a)]
\item $\SH$ is $S$-invariant.
\item Let $\nu\in\widetilde{\MW}$ and assume that $\nu'\in\ocl{\nu}:=\overline{\{S_g\nu:g\in G\}}\subseteq \widetilde{\MW}$.
Then $\SH(\nu')\subseteq\SH(\nu)$.
\item $\overline{\inn(W)}=\SH(\nu)$ for all $\nu\in\Mmin$.
\item $\overline{\inn(W)}\subseteq \SH(\nu)$ and $\inn(W)=\inn(\SH(\nu))$ for all $\nu\in\MW$.
\end{compactenum}
\end{lemma}

\begin{proof}
a) $\SH(S_g(\nu))=\supp(\piH_*(S_g(\nu)))=\supp(\piH_*\nu)=\SH(\nu)$ for all $g\in G$.\\
b) If $\nu=\0$ or $\nu'=\0$, the claim is trivial. Otherwise, the claim follows immediately from Lemma~\ref{lem:lsc}.\\
c) Assertion b) applies to any two $\nu,\nu'\in\Mmin$. Hence $\SH$ is constant on $\Mmin=\overline{\nuW(\CW)}$. 
But for any continuity point $\hx$  of $\nuW$ we have $\SH(\nuW(\hx))=\overline{\inn(W)}$ by Remark~\ref{rem:SHcont}.\\
d) 
Let $\nu\in\MW$. As $\Mmin$ is the unique minimal subset of $\MW$, we have $\Mmin\subseteq\ocl\nu$. Let $\nu'\in\Mmin$. Then $\overline{\inn(W)}=\SH(\nu')$ by part c), and $\SH(\nu')\subseteq\SH(\nu)$ by part a), because $\nu'\in\Mmin\subseteq\ocl\nu$. Hence $\inn(W)\subseteq\overline{\inn(W)}\subseteq\SH(\nu)\subseteq W$, in particular also $\inn(W)\subseteq\inn(\SH(\nu))\subseteq\inn(W)$.
\end{proof}

\begin{lemma}\label{lemma:shifted-versions}
Suppose $\nu,\nu'\in\widetilde{\MW}$ and $\piG_*\nu=\piG_*\nu'$. Then there is $d\in H$ such that $\nu'=\sigma_d\nu$, where $(\sigma_d\nu)(A):=\nu(A-(0,d))$ for all Borel subsets $A$ of $G\times H$.
In particular, $d+\SH(\nu')=\SH(\nu)$ and $d+\inn(W)=\inn(W)$.
\end{lemma}
\begin{proof}\emph{(inspired by the proof of \cite[Lem.~4.5]{KR2015})}\\
If $\piG_*\nu=\piG_*\nu'=\0$,
then also $\nu=\0=\nu'$, and the claim is obvious. 
Otherwise, by definition of $\widetilde{\MW}$,
there are $x,x'\in G\times H$ such that $\nu\leqslant\nuW(x+\LL)$ and $\nu'\leqslant\nuW(x'+\LL)$. Hence, $\piG_*\nu\leqslant\piG_*\nuW(x+\LL)$ and
$\piG_*\nu'\leqslant\piG_*\nuW(x'+\LL)$, and as $\0\neq\piG_*\nu=\piG_*\nu'$,
there is $\tilde{\ell}\in\LL$ such that $x_G'=x_G^{}+\tilde{\ell}^{}_G$ by \cite[Lem.~7.1d]{KR2015}.
Let $\tilde{x}=x+\tilde\ell$. Then $\tilde{x}^{}_G=x'_G$ and $\nuW(x+\LL)=\nuW(\tilde{x}+\LL)$.
 Hence we can assume without loss of generality that $x^{}_G=x_G'$.
  As $\piG|_\LL$ is \oneone, we conclude that the following chain of equivalences holds for each $\ell\in\LL$:
\begin{equation*}
\nu\{x+\ell\}=1
\quad\Leftrightarrow\quad
\piG_*\nu\{x^{}_G+\ell^{}_G\}=1
\quad\Leftrightarrow\quad
\piG_*\nu'\{x'_G+\ell^{}_G\}=1
\quad\Leftrightarrow\quad
\nu'\{x'+\ell\}=1\ .
\end{equation*}
For $d:=x'_H-x^{}_H$ this can be rewritten as
\begin{equation*}
\nu\{x'+\ell-(0,d)\}=1
\quad\Leftrightarrow\quad
\nu'\{x'+\ell\}=1\ ,
\end{equation*}
and as the measures $\nu$ and $\nu'$ are sums of unit point masses supported by the sets $x+\LL$ and $x'+\LL$, respectively, $\nu'=\sigma_d\nu$ follows at once. Hence $\supp(\piH_*(\nu'))=\supp(\piH_*(\nu))-d$, so that $\SH(\nu)=d+\SH(\nu')$. Observing Lemma~\ref{lemma:SH-inv}d, this implies $\inn(W)=d+\inn(W)$.
\end{proof}
\begin{proof}[{\bf \em Proof of Theorem~\ref{theo:inn-aperiodic}}]  a)
As $\inn(W)$ is aperiodic, it is in particular nonempty. Hence $(\MW,S)$ is an almost \oneone extension of its maximal equicontinuous factor $(\hX,\hT)$ by \cite[Thm.~1a]{KR2015}. As $\piG_*:(\MW,S)\to(\MWG,S)$
is a continuous factor map between compact systems, all we have to show is that $\piG_*$ is $\oneone$. So let $\nu,\nu'\in\MW$ and suppose that $\piG_*(\nu)=\piG_*(\nu')$. Then
$\nu'=\sigma_d\nu$ and $d+\SH(\nu')=\SH(\nu)$ for some $d\in H$ by Lemma~\ref{lemma:shifted-versions}. In particular $d+\inn(\SH(\nu'))=\inn(\SH(\nu))$, so that $d+\inn(W)=\inn(W)$ by Lemma~\ref{lemma:SH-inv}d.
By assumption, $\inn(W)$ is aperiodic. Therefore $d=0$ and hence $\nu'=\sigma_0\nu=\nu$.\\
b) As $(\MWG,S)$ and $(\MW,S)$ are isomorphic by part a), all results 
for $(\MW,S)$ from \cite{KR2015} apply to $(\MWG,S)$ as well.
In particular, $M$ contains the unique minimal invariant subset $\MminG$ of $\MWG$,
 and, just as $(\MWG,S)$ itself, $(\MminG,S)$ is an almost \oneone extension of $(\hX,\hT)$ with factor map $\Gamma\,|_{\MminG}$ \cite[Thm.~1a]{KR2015}. As $\MminG\subseteq M\subseteq\MWG$, the claim of the theorem follows.
\end{proof}

For later use we continue with some further lemmas highlighting properties of $\SH$.

\begin{lemma}\label{lemma:SH-inv-ergod}
Let $P,P'$ be ergodic $S$-invariant probability measures on $\MW$.
\begin{compactenum}[a)]
\item There
is a Haar regular subset $\WP\subseteq W$ of $H$ such that 
$\SH(\nu)=\WP$ for $P$-a.a. $\nu$. 
It is empty if and only if $P(\{\0\})=1$.
\item If $P\circ(\piG_*)^{-1}=P'\circ(\piG_*)^{-1}$, then $W_{P}=W_{P'}+d$ for some $d\in H$.
\end{compactenum}
\end{lemma}

\begin{proof}
a)
The claim is obvious for $P$ satisfying $P(\{\0\})=1$, which is equivalent to $W_P=\emptyset$. Hence we assume without loss of generality that $P(\MW\setminus\{\0\})=1$.

We need the following preparation. Fix a complete metric $d$ on $H$ that generates the topology, and also  a countable dense subset $\{h_n: n\in\N\}$ of $H$. Define functions $\delta_n:\MW\setminus\{\0\}\to\R$, $\delta_n(\nu):=d(h_n,\SH(\nu))$. These functions inherit their Borel-measurability from $\SH$ \cite[Thm.~III.2 and III.9]{CV1977}. Note also that if $(\delta_n(\nu))_{n\in\N}=(\delta_n(\nu'))_{n\in\N}$, then $\SH(\nu)=\SH(\nu')$. Indeed, otherwise there is $h\in\SH(\nu)\setminus\SH(\nu')$ (w.l.o.g.), and for a subsequence $(h_{n_k})_k$ converging to $h$ one has $\lim_{k\to\infty}\delta_{n_k}(\nu)=d(h,\SH(\nu))=0<d(h,\SH(\nu'))=\lim_{k\to\infty}\delta_{n_k}(\nu')$, a contradiction.

Now the $S$-invariance of $\SH$ implies at once that all $\delta_n$ are $S$-invariant. As $P$ is ergodic, there are constants $(a_n)_{n\in\N}$ 
and a set $\cM'\subseteq\MW\setminus\{\0\}$ of full $P$-measure such that $\delta_n(\nu)=a_n$ for all $\nu\in\cM'$ and all $n\in\N$. Hence $\SH(\nu)$ is the same compact subset of $H$, call it $\WP$, for all $\nu\in\cM'$. We have $W_P\neq\emptyset$, of course.

It remains to prove that $\WP$ is Haar regular. Suppose for a contradiction that this is not the case. Then there are $h\in \WP$ and $r>0$ such that $m_H\left(B_r(h)\cap \WP\right)=0$, where $B_r(h)=\{h'\in H:d(h',h)<r\}$. In view of (\ref{eq:SH(nu)}), $\piH\left(\supp(\nu)\right)\cap B_r(h)\neq\emptyset$ for each $\nu\in\cM'$. 
Using the torus map $\hat\pi: \MW\setminus\{\0\}\to\hX$, which was explained after Proposition~\ref{prop:MWfactor}, we infer $\piH\left(\supp(\nuW(\hat\pi(\nu)))\right)\cap B_r(h)\neq\emptyset$. Denote by $\pi(\nu)$ the unique representative of $\hat\pi(\nu)$ in the fundamental domain $X\subseteq G\times H$ of $\hX$. 
It follows that
\begin{equation*}
\cM'
\subseteq
\bigcup_{\ell\in\LL}
\left\{\nu\in\MW: \piH(\ell+\pi(\nu))\in B_r(h)\cap \WP\right\}
=
\bigcup_{\ell\in\LL}
\pi^{-1}\left(X\cap\bigg(\big(G\times(B_r(h)\cap \WP)\big)-\ell\bigg)\right).
\end{equation*}
In the remaining part of the proof we will show that $P(\cM')=0$, which is the desired contradiction. To that end recall that $\LL$ is countable and that $P\circ\hat\pi^{-1}=m_\hX$, compare the proof of Proposition~\ref{prop:factor}.
Hence $P\circ\pi^{-1}=\dL\cdot m_{G\times H}|_X$, compare Fact~\ref{en:ass}(2).
Therefore it suffices to estimate
\begin{equation*}
\begin{split}
m_{G\times H}\left(X\cap\bigg(\big(G\times(B_r(h)\cap \WP)\big)-\ell\bigg)\right)
&\leqslant
m_{G\times H}\bigg(\big(G\times(B_r(h)\cap \WP)\big)-\ell\bigg)\\
&=
m_{G\times H}\big(G\times(B_r(h)\cap \WP)\big)\ ,
\end{split}
\end{equation*}
and to observe that the latter expression evaluates to $0$, because $m_H(B_r(h)\cap \WP)=0$.
\\
b) Assume now that $P\circ(\piG_*)^{-1}=P'\circ(\piG_*)^{-1}$. In view of part a) of the lemma and of Lemma~\ref{lemma:shifted-versions}, there are $\nu,\nu'\in\MW$ 
and $d\in H$
such that 
$W_P=\SH(\nu)=d+\SH(\nu')=d+W_{P'}$.
\end{proof}

\begin{lemma}\label{lem:piG11}
Suppose that $W$ is aperiodic. Then $\piG_*:\MW\to\MWG$ is \oneone at $\nu\in\MW$ whenever $\SH(\nu)=W$.
\end{lemma}
\begin{proof}
If $\SH(\nu)=W$ and $\piG_*\nu=\piG_*\nu'$ for $\nu,\nu'\in\MW$,
then $\nu'=\sigma_d\nu$ and $W-d=\SH(\nu)-d=\SH(\nu')\subseteq W$ by Lemma~\ref{lemma:shifted-versions}. 
Let $W_n=W-n\cdot d$. Then $W_0\supseteq W_1\supseteq\dots$ is a nested sequence of compact sets. Suppose for a contradiction that there exists some $h_0\in W_0\setminus W_1$ and let $h_n:=h_0-n\cdot d\in W_n\subseteq W$ for $n\in\N$. Then, for $k>n$, $d(h_n,h_k)\geqslant d(h_n,W_k)\geqslant d(h_n,W_{n+1})=d(h_0,W_1)>0$, which is impossible, because $W$ is compact. Hence $W=W_0=W_1=W-d$, so that $d=0$ because $W$ is aperiodic. 
\end{proof}

\begin{lemma}\label{lemma:Borel-measurable-strong}
Suppose that $F\subseteq W$ is Haar regular and $K\subseteq\MW$ is closed. Then $\piG_*\{\nu\in K: \SH(\nu)=F\}$ is a Borel subset of $\MWG$.
\end{lemma}

\begin{proof}
The case $F=\emptyset$ is trivial. So we may assume that $F\neq\emptyset$. 
Denote by $V_1,V_2,\ldots\subseteq H$ those elements of a base of the second countable space $H$, for which $F_n:= F\setminus V_n$ is a proper subset of $F$.
Then any compact proper subset $F'$ of $F$ is contained in some  $F_n$.
We thus can write
\begin{equation}\label{eq:G-new-1}
\piG_*\{\nu\in K:\SH(\nu)\subseteq F\}
=
\piG_*\{\nu\in K:\SH(\nu)=F\}
\cup
\left(\bigcup_{n=1}^\infty\piG_*\{\nu\in K:\SH(\nu)\subseteq F_n\}\right) .
\end{equation}
Next,
let $\nu,\nu'\in\MW$ with $\SH(\nu)=F\neq\emptyset$, $\SH(\nu')\subseteq F$ and
$\piG_*(\nu')=\piG_*(\nu)$. In particular, $\nu,\nu'\neq\0$.
Then $m_H(\SH(\nu'))=m_H(\SH(\nu))=m_H(F)$ by Lemma~\ref{lemma:shifted-versions}. 
As $F$ is Haar regular, this implies $\SH(\nu')=F$. Therefore the union of the two sets on the rhs of Eqn.~\eqref{eq:G-new-1} is disjoint, and we have
\begin{equation}\label{eq:G-new-3}
\piG_*\{\nu\in K:\SH(\nu)=F\}
=
\piG_*\{\nu\in K:\SH(\nu)\subseteq F\}
\setminus\bigcup_{n=1}^\infty\piG_*\{\nu\in K:\SH(\nu)\subseteq F_n\} \ .
\end{equation}
As all sets involved in the rhs of \eqref{eq:G-new-3} are continuous images of sets which are compact by Lemma~\ref{lem:lsc}, the lhs of \eqref{eq:G-new-3} is in particular Borel measurable.
\end{proof}

\section{Proofs of Theorems~\ref{theo:main-2} and~\ref{theo:main}}
\label{sec:proof-2-3}

We first prove both theorems for compact windows, and discuss the extension
of Theorem~\ref{theo:main-2} to windows which are compact modulo $0$ in Subsection~\ref{subsec:mtr}.

Fix any tempered van Hove sequence $(A_n)_{n\in\mathbb N}$ of subsets of $G$, compare \cite[Footnote~5]{KR2015}. We always have the upper bound
\begin{equation*}
\overline{d}(\nu):=\limsup_{n\to\infty}\frac{\nu(A_n\times H)}{m_G(A_n)}\le \dL\cdot m_H(W)
\end{equation*}
on the upper density of any configuration $\nu\in \MW$, see \cite[Eqn.~(14)]{KR2015}. We say that $\nu\in\MW$ has \emph{maximal density} if
\begin{equation*}
d(\nu):=\lim_{n\to\infty}\frac{\nu(A_n\times H)}{m_G(A_n)}\text{ exists, and if }d(\nu)=\dL\cdot m_H(W)\,.
\end{equation*}
Recall from \cite[Thm.~5a]{Moody2002} that 
\begin{equation}\label{eq:Moody}
\QM\left\{\nu\in\MW: \text{$\nu$ has maximal density}\right\}=1\,,
\end{equation}
where $\QM=m_\hX\circ(\nuW)^{-1}$ denotes the Mirsky measure on $\MW$.
\begin{lemma}\label{lemma:SHnu=W}
$W_{reg}\subseteq\SH(\nu)\subseteq W$ for
each
$\nu\in\MW$ with maximal density.
\end{lemma}
\begin{proof}
We have $\SH(\nu)\subseteq W$ for all $\nu\in\MW$ by definition. 
If $m_H(W)=0$, then $W_{reg}=\emptyset\subseteq\SH(\nu)$.
Assume now that $m_H(W)>0$ and that $\nu$ has maximal density $d(\nu)=\dL\cdot m_H(W)>0$. In particular, $\nu\in\MW\setminus\{\0\}$. There is a unique $\hx=x+\LL\in\hX$ such that $\nu\leqslant\nuW(\hx)$ by \cite[Lem.~5.4]{KR2015}.
We thus get $\supp(\nu) \subseteq (x+\LL) \cap (G\times \SH(\nu))$,
which implies $\nu\leqslant\nu_{\scriptstyle \SH(\nu)}(\hx)$. Hence
\begin{equation*}
\dL\cdot m_H(W)
=d(\nu)
\leqslant \overline{d}(\nu_{\scriptstyle \SH(\nu)}(\hx))
\leqslant
\dL\cdot m_H(\SH(\nu))\ ,
\end{equation*}
which yields $m_H(W)= m_H(\SH(\nu))$.
As $\SH(\nu)$ is a compact subset of $W$, this implies
$W_{reg}\subseteq \SH(\nu)$.
\end{proof}
\begin{corollary}\label{coro:WQM=Wreg}
$W_{\QM}=W_{reg}$.
\end{corollary}
\begin{proof}
Observe that $W_{\QM}=\SH(\nu)\subseteq W$ for $\QM$-a.a.~$\nu$ by Lemma~\ref{lemma:SH-inv-ergod}. Hence 
$W_{reg}\subseteq W_{\QM}$ in view of \eqref{eq:Moody} and Lemma~\ref{lemma:SHnu=W}.  Haar regularity of $W_{\QM}$, which holds by definition,  implies $W_{reg} \supseteq W_{\QM}$. Indeed, let $w\in W_{\QM}$ and let $U$ any neighborhood of $w$. Then $m_H(U\cap W)=m_H(U\cap W_{\QM})>0$, because $m_H(W\setminus W_{\QM})\leqslant m_H(W\setminus W_{reg})=0$ and as $W_{\QM}$ is Haar regular.
\end{proof}

\begin{corollary}\label{coro:max-density}
Denote by $\MW'\subseteq \MW$ the set of configurations of maximal density.  If $W$ is Haar aperiodic, then $\piG_*|_{\MW'}:\MW'\to\MWG$ is \oneone.
\end{corollary}
\begin{proof}
Note that $W_{reg}\subseteq \SH(\nu)\subseteq W$ for any configuration $\nu\in\MW'$ by Lemma~\ref{lemma:SHnu=W}, where $m_H(W\setminus W_{reg})=0$. Now assume that $\piG_*(\nu)=\piG_*(\nu')$ for $\nu,\nu'\in \MW'$. Then $\nu'=\sigma_d \nu$ and $d+\SH(\nu')=\SH(\nu)$ for some $d\in H$ by Lemma~\ref{lemma:shifted-versions}, which implies
\begin{displaymath}
0=m_H((d+\SH(\nu'))\triangle \SH(\nu))=m_H((d+W)\triangle W).
\end{displaymath}
As $W$ is Haar aperiodic, we conclude $d=0$ and thus $\nu=\nu'$.
\end{proof}

\begin{proof}[\bf \em Proof of Theorem~\ref{theo:main-2}]  
$\piG_*$ is \oneone at $\QM$-a.a. $\nu\in\MW$ by Corollary~\ref{coro:max-density} and Eqn.~\eqref{eq:Moody}.
Hence $\piG_*:(\MW,\QM,S)\to(\MWG,\QMG,S)$ is a measure-theoretic isomorphism, and thus $\nuWG:(\hX,m_\hX, \hT)\to(\MWG,\QMG,S)$ is a measure-theoretic isomorphism by Proposition~\ref{prop:mirsky}. Here we use $m_H(W)>0$, which follows from Haar aperiodicity of $W$.
\end{proof}
We now turn to general $S$-invariant probability measures on $\MW$.
\begin{corollary}\label{coro:ergodic-oneone}
Fix an ergodic $S$-invariant probability measure $P$ on $\MW$ and consider the Haar regular set $\WP\subseteq W$ from Lemma~\ref{lemma:SH-inv-ergod}.
Then $\cM_P:=\SH^{-1}\{\WP\}\subseteq \MW$ has $P$-measure one.  If $\WP$ is aperiodic, then $\piG_*|_{\cM_P}$ is \oneone.
\end{corollary}
\begin{proof}
If $P=\delta_{\0}$, then $W_P=\emptyset$ and $\cM_P=\{\0\}$, and the claim is trivial. Otherwise we may assume that $P(\MW\setminus\{\0\})=1$. Then by Lemma~\ref{lem:lsc}, the set $\cM_P$ is measurable. By Lemma~\ref{lemma:SH-inv-ergod} we have $P(\cM_P)=1$. The injectivity of 
$\piG_*|_{\cM_P}$ follows from Lemma~\ref{lemma:shifted-versions}, where we use that $\WP$ is aperiodic.
\end{proof}

In order to infer results on $\MWG$ from $\MW$, we need to ``lift'' invariant probability measures from $\MWG$ to $\MW$. This is the content of the following proposition.  

\begin{proposition}\label{prop:main}
Let $\PG$ be an ergodic $S$-invariant probability measure on $\MWG$, and denote by $\cP(\PG)$ the family of all $S$-invariant probability measures $P$ on $\MW$ that project to $\PG$, i.e., for which $P\circ(\piG_*)^{-1}=\PG$. Then the following hold.
\begin{compactenum}[a)]
\item $\cP(\PG)\neq\emptyset$.
\item  Each  $P\in\cP(\PG)$ has an ergodic decomposition $P=\int P_e\circ\sigma_h^{-1}\ d\rho(h)$ for some compactly supported probability measure $\rho$ on $H$ and some ergodic $S$-invariant probability measure $P_e\in\cP(\PG)$.
\item If $P\in\cP(\PG)$ is ergodic and if $\WP\subseteq W$ is aperiodic, then $\piG_*:\MW\to\MWG$ is a measure-theoretic isomorphism between $(\MW,P,S)$ and $(\MWG,\PG,S)$, and both systems are extensions of $(\hX,m_\hX,\hT)$.
\end{compactenum}
\end{proposition}
\begin{proof}
a)\ 
Denote by $\Gamma$ the set valued map (also called multifunction) from $\MWG$ to compact subsets of $\MW$ defined by $\Gamma(\nuG)=(\piG_*)^{-1}\{\nuG\}$.
It is measurable in the sense of \cite[Thm.~III.2]{CV1977}, because
$\Gamma^-(C):=\{\nuG\in\MWG:(\piG_*)^{-1}\{\nuG\}\cap C\neq\emptyset\}=\piG_*(C)$
is compact and hence Borel measurable for any closed $C\subseteq\MW$.
Hence, by the measurable selection theorem \cite[Thm.~III.6]{CV1977}, there is a 
Borel measurable map
$\psi_*^{}:\MWG\to\MW$ such that $\piG_*\circ\psi^{}_*=\id_{\MWG}$.
In particular, $P:=\PG\circ (\psi^{}_*)^{-1}$ is a well defined probability measure on $\MW$ that projects to $\PG$. The measure $P$ is not necessarily $S$-invariant,  but a Krylov-Bogoliubov construction on $P$, see e.g.~\cite[Thm.~8.10]{EinWard2011}, provides an $S$-invariant probability measure in $\cP(\PG)$. The latter holds since  also $(P\circ S_g^{-1})$ projects to $\PG$ for every $g\in G$ as  $(P\circ S_g^{-1})\circ(\piG_*)^{-1}=P\circ(\piG_*\circ S_g)^{-1}=P\circ(S_g\circ\piG_*)^{-1}=P\circ(\piG_*)^{-1}\circ S_g^{-1}=\PG\circ S_g^{-1}=\PG$, and  since the measure transport by $\piG_*$ is continuous w.r.t. the weak topology.
\\
b)\ 
For any $P\in\cP(\PG)$, its ergodic decomposition \cite[Thm.~8.20]{EinWard2011} can be written as 
\begin{equation}\label{eq:erg-decomp-1}
P=\int_{\MW} P_\mu\ dP(\mu)\ ,
\end{equation}
where the probability measures $P_\mu$ on $\MW$ are ergodic and where $\mu\mapsto P_\mu$ is Borel measurable. Then
\begin{equation*}
\PG=P\circ(\piG_*)^{-1}=
\int_{\MW} P_\mu\circ(\piG_*)^{-1}\ dP(\mu)\ ,
\end{equation*}
where all $P_\mu\circ(\piG_*)^{-1}$ are ergodic $S$-invariant measures on $\MWG$.
As $\PG$ itself is ergodic, 
$\PG=P_\mu\circ(\piG_*)^{-1}$ for $P$-a.a. $\mu$, so that
$P_\mu\in\cP(\PG)$ for $P$-a.a. $\mu$.

Fix any measure $P_e$ from the ergodic decomposition. Let $P_\mu$ be any other measure from this decomposition. $P_e$ and $P_\mu$ can be disintegrated over $\PG$, namely there are systems $\{p_{\nuG}:\nuG\in\MWG\}$ and $\{p_{\nuG}':\nuG\in\MWG\}$  of probability measures  on $\MW$, such that
\begin{equation*}
P_e=\int_{\MWG}p_{\nuG}\ d\PG(\nuG)
\quad\text{and}\quad
P_\mu=\int_{\MWG}p'_{\nuG}\ d\PG(\nuG)\ .
\end{equation*}
Hence, for $\PG$-a.a. $\nuG\in\MWG$ and $(p_{\nuG}\otimes p_{\nuG}')$-a.a. $(\nu,\nu')\in\MW\times\MW$, we have $\piG_*\nu=\nuG=\piG_*\nu'$, so that $\nu'=\sigma_d\nu$ for some $d=d(\nu,\nu')\in H$. Since we may assume that $\nu$ and $\nu'$ are generic for 
$P_e$ and $P_\mu$, respectively, we can conclude
$P_\mu=P_e\circ\sigma_d^{-1}$. 

Consider the map $\kappa:H\to\cP(\PG),\ h\mapsto P_e\circ\sigma_h^{-1}$. It is continuous so that $\kappa(H)$ is a compact subset of the space of probability measures on $\MW$. Observe that the set $\kappa(H)$ does not depend on the choice of a particular $P_e$ in the definition of~$\kappa$. As
$P_\mu\in\kappa(H)$ for $P$-a.a. $\mu$, we can rewrite the ergodic decomposition (\ref{eq:erg-decomp-1}) as
\begin{equation}
P=\int_{\kappa(H)}\tilde P\ d\tilde\rho(\tilde P)\ ,
\end{equation}
where $\tilde\rho$ is the distribution of the random measures $P_\mu$ under $P$.

The set valued map $\Gamma:P\mapsto\kappa^{-1}(\{P\})$ from $\cP(\PG)$ to compact subsets of $H$ is Borel measurable by the same arguments as in part a) of the proof.
Hence, by the measurable selection theorem \cite[Thm.~III.6]{CV1977}, there is a 
Borel measurable map
$\kappa^\dagger:\kappa(H)\to H$ such that $\kappa\circ\kappa^\dagger=\id_{\kappa(H)}$.
In particular, $\rho:=\tilde\rho\circ (\kappa^\dagger)^{-1}$ is a well defined probability measure on $H$, and
\begin{equation}
P=\int_{\kappa(H)}\kappa(\kappa^\dagger(\tilde P))\ d\tilde\rho(\tilde P)
=\int_{H}\kappa(h)\ d\rho(h)
=\int_{H}P_e\circ\sigma_h^{-1}\ d\rho(h)
\ .
\end{equation}
c)\ 
If $P\in\cP(\PG)$ is ergodic and if $\WP$ is aperiodic, then
$\piG_*:(\MW,P,S)\to(\MWG,\PG,S)$ is a measure-theoretic isomorphism in view of Corollary~\ref{coro:ergodic-oneone}, and both systems are extensions of $(\hX,m_\hX,\hT)$ by Proposition~\ref{prop:factor}. Here we use $m_H(W)\ge m_H(W_P)>0$, as the aperiodic Haar regular set $W_P$ is Haar aperiodic.
\end{proof}

\section{Periodic windows and quotient cut-and-project schemes}\label{sec:periods}

For a given a subset $A\subseteq H$, recall its period group $H_A=\left\{h\in H: h+A=A\right\}$. The set $A\subseteq H$ is (topologically) aperiodic, if $H_A=\{0\}$.
\begin{lemma}\label{lem:aper}
\begin{compactenum}[a)]
\item $H_A\subseteq H_{\bar A}\cap H_{\inn(A)}$.
\item If $A$ is closed, then also $H_A$ is closed.
\item If $\inn(\bar A)= \inn(A)$ (e.g. if
$A$ is closed), then $H_{\inn(A)}=H_{\overline{\inn(A)}}$.
\item If $A$ is compact and nonempty, then $H_A$ is compact.
\end{compactenum}
\end{lemma}
\begin{proof}
a) For each $h\in H$, the translation by $h$ is a homeomorphism on $H$.\\
b) Let $h_n\in H_A$, $h=\lim_n h_n$. If $w\in A$, then $\pm h+w=\lim_n(\pm h_n+w)\in A$, because all $\pm h_n+w$ are in $A$ and $A$ is closed.  This shows $\pm h +A\subseteq A$, i.e., $h+A=A$. \\
c) The assumption implies $\inn(\overline{\inn(A)})=\inn(A)$. Indeed, $\inn(A)\subseteq \overline{\inn(A)}$ implies $\inn(A)\subseteq \inn(\overline{\inn(A)})$, and $\inn(A)\subseteq A$ implies $\inn(\overline{\inn(A)})\subseteq \inn(\overline{A})\subseteq\inn(A)$. Now the result follows from the implications
\begin{displaymath}
h+\inn(A)=\inn(A)  \Longrightarrow h+\overline{\inn(A)}=\overline{\inn(A)} \Longrightarrow h+\inn(\overline{\inn(A)})=\inn(\overline{\inn(A)}) \ .
\end{displaymath}
d) By definition we have $H_A+A=A$. As $A$ is compact nonempty, $H_A$ must be compact, too.
\end{proof}

\begin{remark}\label{rem:nd-boundary}
The condition $\inn(\overline A)\subseteq \inn(A)$ in part c) of the lemma implies that the topological boundary
$\partial A=\overline A\setminus\inn(A)$ is nowhere dense. Indeed, nowhere denseness is equivalent to the condition $\inn(\overline A)\subseteq \overline{\inn(A)}$. These two conditions are however not equivalent, as is easily seen by looking at open dense sets $A$.
\end{remark}

For a given cut-and-project scheme with window $W\subseteq H$, an important example is the period group $H_W$ of the window. Some structural results for model sets rely on the assumption of an aperiodic window. Aperiodicity may however be assumed without loss of generality by passing to an associated quotient cut-and-project scheme, where the periods of the window have been factored out, compare  \cite[Section 9]{BLM07}. As this construction has not been fully described before, we present it here in some detail.

Let $(G,H,\LL)$ and $\hX=(G\times H)/\LL$ be as before, with quotient map $\pihX:G\times H\to\hX$. Fix any compact subgroup $H_0\subseteq H$ and consider
$H':=H/H_0$ with factor map $\varphi:H\to H'$. Consider $\cH_0:=\{0\}\times H_0\subseteq G\times H$ and note that $\LL\cap\cH_0=\{(0,0)\}$ as $\piG$ is \oneone on $\LL$. Observe next that $\iota:(G\times H)/\cH_0\to G\times H'$, $\iota\left((g,h)+\cH_0\right):=(g,h+H_0)$, is a (rather trivial) isomorphism of topological groups.
Denote by $\Phi$ the quotient map $\Phi:G\times H\to(G\times H)/\cH_0$, and let $\LL':=\iota(\Phi(\LL))$. 

\begin{lemma}\label{lemma:closed-subgroup}
$\LL'=\iota(\Phi(\LL))$ is a discrete subgroup of $G\times H'$.
\end{lemma}
\begin{proof}
As $\iota\circ\Phi$ is a group homomorphism, $\LL'$ is a subgroup of $G\times H$. We prove that $\LL'$ is discrete:
Take a compact zero neighborhood $\mathcal U\subseteq G\times H$ such that $\LL\cap \mathcal U=\{0\}$, which is possible since $\LL$ is discrete in $G\times H$. Now $\LL'\cap \iota(\mathcal U+\cH_0)$ contains $0$ and is finite as $\mathcal U+\cH_0\subseteq G\times H$ is compact. Hence there is a zero neighborhood $\mathcal V\subseteq \mathcal U$ such that $\LL' \cap \iota(\mathcal V+\cH_0)=\{0\}$.
\end{proof}

This lemma allows to consider the locally compact abelian quotient group $\hXprime=(G\times H')/\LL'$. Now
$(G,H',\LL')$ is a cut-and-project scheme with associated torus $\hXprime$,  compare \cite[Section 9]{BLM07}, which we call a quotient cut-and-project scheme. For the convenience of the reader, we give a proof which is based on the following general facts about quotient groups.

\begin{lemma}\cite[Proposition III.20]{Husain1966}\label{lemma:isomorphisms}\\
$(G\times H)/(\LL+\cH_0)$ is isomorphic (as a topological group) to each of the following groups:
\begin{compactenum}[a)]
\item $\hX/\pihX(\LL+\cH_0)=\hX/\pihX(\cH_0)$
\item $((G\times H)/\cH_0)/\Phi(\LL+\cH_0)=((G\times H)/\cH_0)/\Phi(\LL)$
\end{compactenum}
\end{lemma}

\begin{corollary}\label{coro:iso-quotients}
In the above setting,  $(G, H',\LL')$ is a cut-and-project scheme, 
in particular $\LL'$ is cocompact.
The topological quotient group $\hXprime=(G\times H')/\LL'$ is isomorphic to $\hX/\pihX(\cH_0)$. 
\end{corollary}
\begin{proof}
Projection properties are inherited: Assume that $0=\piG((\iota\circ\Phi)(\ell))=\ell_G$ for some $\ell\in \LL$. As $\piG$ is \oneone on $\LL$, we infer $\ell=0$, which implies $(\iota\circ\Phi)(\ell)=0$. Note also that $\overline{\piHprime(\LL')}=\overline{\varphi(\piH(\LL))}\supseteq \varphi(\overline{\piH(\LL)})=\varphi(H)=H'$. Here we used continuity and surjectivity of the projection map, together with the assumption that $\piH(\LL)$ is dense in $H$.\\
As $G\times H'=\iota\left((G\times H)/\cH_0\right)$ and $\LL'=\iota(\Phi(\LL))$, where $\iota$ is an isomorphism, we see that
$(G\times H')/\LL'$ is isomorphic to $((G\times H)/\cH_0)/\Phi(\LL)$.
Combining this with a) and b) of Lemma~\ref{lemma:isomorphisms}, we conclude that $(G\times H')/\LL'$ is isomorphic to $\hX/\pihX(\cH_0)$.
As $\cH_0$ is compact, $\pihX(\cH_0)$ is compact, so that $\hX/\pihX(\cH_0)$ is compact \cite[Theorem III.11]{Husain1966}.
In particular, $\LL'$ is cocompact in $G\times H'$.
\end{proof}
\begin{remark}\label{remark:Haar-measure-factor}
The factor map $\iota\circ\Phi:G\times H\to G\times H'$ carries over to a factor map $\widehat{\iota\Phi}:\hX\to\hXprime$, because $\iota(\Phi(x+\LL))=\iota(\Phi(x))+\LL'$.
It pushes the Haar measure $m_\hX$ to the Haar measure $m_{\hXprime}$ on $\hXprime$.
\end{remark}

If
$W\subseteq H$ is a window in $(G,H,\LL)$ and $H_0$ is a closed subgroup of $H$, we can consider the quotient cut-and-project scheme $(G,H',\LL')$ with window $W':=\varphi(W)$, because $W'$ inherits the basic topological properties from $W$:

\begin{lemma}\label{lemma:Wprime}
\begin{compactenum}[a)]
\item $W'$ is compact.
\item If $W$ is topologically regular, then so is $W'$.
\item If $W$ is Haar regular, then so is $W'$.
\end{compactenum}
\end{lemma}
\begin{proof}
a) Recall that $\varphi:H\to H'$ is continuous and open \cite[Theorem III.10]{Husain1966}.
In particular, $W'=\varphi(W)$ is compact. 

b) If $W$ is topologically regular, i.e. if 
$W=\overline{\inn(W)}$, then 
\begin{equation*}
W'=\varphi(\overline{\inn(W)})\subseteq\overline{\varphi(\inn(W))}\subseteq\overline{\inn(\varphi(W))}=\overline{\inn(W')}\ ,
\end{equation*}
where we used continuity of $\varphi$ for the first inclusion and openness for the second one. 

c) Suppose now that $W$ is Haar regular, and let $U\subseteq H'$ be open with $m_{H'}(U\cap W')=0$. We must show that $U\cap W'=\emptyset$. By definition of the quotient topology, $\varphi^{-1}(U)$ is open in $H$. As $\varphi^{-1}(U)\cap W\subseteq\varphi^{-1}(U\cap W')$, we have $m_H(\varphi^{-1}(U)\cap W)\leqslant m_H(\varphi^{-1}(U\cap W'))=m_{H'}(U\cap W')=0$, so that $\varphi^{-1}(U)\cap W=\emptyset$, because $W$ is Haar regular. This implies $U\cap W'= U\cap\varphi(W)=\emptyset$.
\end{proof}

 An important example is $H_0:=H_W$, the period group of a window, because $W'$ is aperiodic in this case.
We will study the relations between the two cut-and-project schemes $(G,H,\LL)$ and $(G',H',\LL')$ with associated windows $W$ and $W'$, respectively. For clarity we write $\cH_W$ instead of $\cH_0$ and add the index $W$ also to the quotient maps. Recall that the quotient map $\iota\circ \Phi_W: G\times H\to G\times H'$ is given by $x\mapsto (\iota\circ\Phi_W)(x)=\iota(x+\cH_W)=(x_G, x_H+ H_W)$.

\begin{lemma}\label{lemma:Wprime-aperiodic}
The window $W'=\varphi_W(W)$ is an aperiodic subset of $H'$.
\end{lemma}
\begin{proof}
Suppose that $W'+h'=W'$ for some $h'=\varphi_W(h)\in H'$. Then 
$\varphi_W(W+h)=\varphi_W(W)$,
in particular $W+h+H_W=W+H_W$. As $W+H_W=W$, this shows that $h\in H_W$, so that $h'=\varphi_W(h)$ is the neutral element of $H'$.
\end{proof}

\begin{lemma}\label{lem:quotientnu}
Let $x\in G\times H$ and $x'\in G\times H'$ be such that $x'=(\iota\circ \Phi)(x)$, and let $W'=\varphi_W(W)$.
\begin{compactenum}[a)]
\item  The sets $(x+\LL)\cap (G\times W)$ and $(x'+\LL')\cap (G\times W')$ are in $1-1$ correspondence  via the quotient map $\iota\circ \Phi$. In particular we have
$\piG((x+\LL)\cap (G\times W))=\piG((x'+\LL')\cap (G\times W'))$.
\item For every $h\in H_W$ we have $\nu_W(x+(0,h)+\LL)=\sigma_h\nuW(x+\LL)$, where $\sigma_h\nu(A)=\nu(A-(0,h))$.
\item We have $\nuWprime(x'+\LL')=\sum_{y\in (x+\LL)\cap (G\times W)}\delta_{\iota(y+\cH_W)}$, which implies $\piG_*(\nuWprime(x'+\LL'))=\piG_*(\nuW(x+\LL))$, i.e.
$\nuWGprime(x'+\LL')=\nuWG(x+\LL)$.
\end{compactenum}
\end{lemma}

\begin{proof}
a) To show injectivity of the quotient map, assume that $(\iota\circ\Phi)(x+\ell_1)=(\iota\circ\Phi)(x+\ell_2)$, that is $x+\ell_1+\cH_W=x+\ell_2+\cH_W$. We thus can conclude $\ell_1-\ell_2\in \{0\}\times H_W$. Hence $\ell_{1,G}-\ell_{2,G}=0$, and as $\piG$ is  \oneone on $\LL$, we infer $\ell_{1,H}-\ell_{2,H}=0$. Hence $\ell_1=\ell_2$, and we get $x+\ell_1=x+\ell_2$.
To show that the quotient map is onto, assume without loss of generality that $W$ is nonempty. Take arbitrary $y'\in (x'+\LL')\cap(G\times W')$. Then $y'=\iota(x+\ell +\cH_W)$ for some $\ell\in \LL$, which implies $x_H+\ell_H+H_W\in W'=\{w+H_W: w\in W\}$. This 
implies $x_H+\ell_H\in W+H_W=W$,
which means that $y:=x+\ell\in (x+\LL)\cap (G\times W)$. The remaining statement is now obvious, since we have $\piG(x)=\piG((\iota\circ\Phi)(x))$.\\
b) Let $h\in H_W$ be given. A direct calculation yields 
\begin{displaymath}
\nuW(x+(0,h)+\LL)=\sum_{y\in (x+(0,h)+\LL)\cap (G\times W)}\delta_y=\sum_{y\in (x+\LL)\cap (G\times W)}\delta_{y+(0,h)}=\sigma_h \nuW(x+\LL).
\end{displaymath}
c) A direct calculation yields
\begin{displaymath}
\nuWprime(x'+\LL')=\sum_{y'\in (x'+\LL')\cap (G\times W')}\delta_{y'}
=\sum_{\iota(y+\cH_W)\in (x'+\LL')\cap (G\times W')}\delta_{\iota(y+\cH_W)}
=\sum_{y\in (x+\LL)\cap (G\times W)}\delta_{\iota(y+\cH_W)},
\end{displaymath}
where we use a) for the third equality. The remaining statement is now obvious.
\end{proof}

\begin{proposition}\label{prop:same-images}
Let $W'=\varphi_W(W)$.
\begin{compactenum}[a)]
\item  $\MWG=\piG_*(\MW)=\piG_*(\MWprime')=\MWGprime$.
\item $\QMG=m_\hX\circ(\nuWG)^{-1}=m_{\hXprime}\circ(\nuWGprime)^{-1}=\QMGprime$
\end{compactenum}
\end{proposition}

\begin{proof}
a) Note the following chain of equivalences:
\begin{displaymath}
\begin{split}
\nuG\in \piG_*(\nuW(\hX)) &
\Leftrightarrow \exists x\in G\times H: \nuG=\piG_*(\nuW(x+\LL))\\
&\Leftrightarrow \exists x'\in G\times H': \nuG=\piG_*(\nuWprime(x'+\LL'))
\\
&\Leftrightarrow \nuG\in \piG_*(\nuWprime(\hXprime)),
\end{split}
\end{displaymath}
where we used Lemma~\ref{lem:quotientnu} c). This means that $\piG_*(\nuW(\hX))=\piG_*(\nuWprime(\hXprime))$. Now a) of the proposition follows from continuity of $\piG_*$ and compactness of $\MW=\overline{\nuW(\hX)}\subseteq \cM$ and $\MWprime'=\overline{\nuWprime(\hXprime)}\subseteq \cM'$, the space of locally finite measures on the Borel subsets of $G\times H'$. \\
b) In view of Lemma~\ref{lem:quotientnu}c,
$(\nuWGprime\circ\widehat{\iota\Phi})(x+\LL)=\nuWGprime(\iota(\Phi(x))+\LL')=\nuWG(x+\LL)$ for all $x\in G\times H$. Hence, observing
Remark~\ref{remark:Haar-measure-factor}, 
$m_{\hXprime}\circ(\nuWGprime)^{-1}=m_\hX\circ\widehat{\iota\Phi}^{-1}\circ(\nuWGprime)^{-1}=m_\hX\circ (\nuWG)^{-1}$.
\end{proof}

\section{Proofs of Theorems~\ref{theo:inn-periodic}, \ref{theo:main-2-periodic} and \ref{theo:main-periodic}}\label{sec:proofs-II}

In this section, $W$ is again a compact window.
We begin with a technical lemma that will be used at several places.
\begin{lemma}\label{lemma:technical}
Let $\nu,\nu'\in\MW$ and $W_0\subseteq W$ be such that 
$\piG_*(\nu')=\piG_*(\nu)$ and
$\SH(\nu)=\SH(\nu')=W_0$. Then $\gamma(\nu')-\gamma(\nu)\in \pihX(\cH_{W_0})$.
\end{lemma}
\begin{proof}
As $\piG_*(\nu)=\piG_*(\nu')$,  Lemma~\ref{lemma:shifted-versions} implies that
$\nu'=\sigma_d\nu$ for some $d\in H_{W_0}$. Recall that $\gamma(\nu)$ equals $x+\LL$ for any point $x\in G\times H$ with $\nu\{x\}=1$. But $\nu\{x\}=\nu'\{x+(0,d)\}$, so $\gamma(\nu')=x+(0,d)+\LL$. Hence $\gamma(\nu')-\gamma(\nu)=(0,d)+\LL\in\pihX(\cH_{W_0})$.
\end{proof}

Suppose that $\inn(W)\neq\emptyset$.
Denote by $\gamma$ the factor map from $\MW$ onto its maximal equicontinuous factor $\hX$
\footnote{This is the map $\pihX_*\circ(\piGH_*)^{-1}$ from \cite[Thm.~1a]{KR2015}.}, and by $\rho$ the factor map from $\hX$ onto $\hX/\pihX(\cH_{\inn(W)})$.
We define a factor map $\Gamma$ from $\MWG$ to $\hX/\pihX(\cH_{\inn(W)})$ as follows: for $\nuG\in\MWG$ pick any $\nu\in(\piG_*)^{-1}\{\nuG\}$ and let $\Gamma(\nuG)=\rho(\gamma(\nu))$.

\begin{lemma}\label{lem:factor}
Suppose that $\inn(W)\neq\emptyset$.
\begin{compactenum}[a)]
\item $\Gamma$ is well defined.
\item $\Gamma$ is continuous and commutes with the dynamics.
\item If $H_{\inn(W)}=H_{W}$ and $\Gamma(\nuG)=\rho(\hx)$ for some $\nuG\in\MWG$ and $\hx\in\CW$, then $\nuG=\piG_*(\nuW(\hx))$.
\item If $H_{\inn(W)}=H_W$, then $\Gamma$ is almost \oneone.
\end{compactenum}
\end{lemma}
\begin{proof}
a) 
Suppose that $\piG_*(\nu)=\piG_*(\nu')$. Then 
$\gamma(\nu')-\gamma(\nu)\in\pihX(\cH_{\inn(W)})$ by Lemma~\ref{lemma:technical}, so that $\rho(\gamma(\nu'))=\rho(\gamma(\nu)+\pihX(\cH_{\inn(W)}))=\rho(\gamma(\nu))$.
This shows that $\Gamma$ is well defined.\\
b) Let $D\subseteq \hX/\pihX(\cH_{\inn(W)})$ be closed. Then 
$E:=\gamma^{-1}(\rho^{-1}(D))$ is closed in $\MW$, and
\begin{equation}\label{eq:Gamma-def}
\nuG\in\Gamma^{-1}(D)
\;\Leftrightarrow\;
\exists\,\nu\in E:\ \piG_*(\nu)=\nuG
\;\Leftrightarrow\;
\nuG\in \piG_*(E)\ .
\end{equation}
As $E$ is also compact, this shows that $\Gamma^{-1}(D)=\piG_*(E)$ is closed. Hence $\Gamma$ is continuous. We show that it commutes with the dynamics: Let $\nuG\in\MWG$, $g\in G$, and denote $D:=\{\Gamma(\nuG)\}$ and $E:=\gamma^ {-1}(\rho^{-1}(D))$. Then $\nuG\in\Gamma^{-1}(D)=\piG_*(E)$ by \eqref{eq:Gamma-def}, and 
\begin{equation*}
S_g\circ\gamma^{-1}\circ\rho^{-1}=(\rho\circ\gamma\circ S_{-g})^{-1}=(\hT_{-g}\circ\rho\circ\gamma)^{-1}=\rho^{-1}\circ\gamma^{-1}\circ \hT_g\ .
\end{equation*}
Hence $S_g\nuG\in S_g(\piG_*(E))=\piG_*(S_g(\gamma^{-1}(\rho^{-1}(D))))=\piG_*(\gamma^{-1}(\rho^{-1}(\hT_g(D))))=\Gamma^{-1}(\hT_g(D))$, where we used again \eqref{eq:Gamma-def} for the last identity. Therefore, $\Gamma(S_g\nuG)\in\hT_g(D)=\{\hT_g(\Gamma(\nuG))\}$.\\
c) Let $h\in H_W$. Then $\hx\in\CW$ if and only if $\hx+(0,h)\in\CW$.  This follows immediately from
\begin{equation}\label{eq:CW-inv}
\begin{split}
\nuW(\hx+(0,h))=\sigma_h(\nuW(\hx))\ ,
\end{split}
\end{equation}
see Lemma~\ref{lem:quotientnu}b). 

Suppose now that $\hx\in\CW$ and $\Gamma(\nuG)=\rho(\hx)$ for some $\nuG\in\MWG$. There is $\nu\in\MW$ such that $\piG_*(\nu)=\nuG$ and $\rho(\gamma(\nu))=\rho(\hx)$. Hence $\gamma(\nu)\in\hx+\cH_{\inn(W)}=\hx+\cH_W$, i.e., there is $h\in H_W$ such that $\gamma(\nu)=\hx+(0,h)\in\CW$. This implies $\nu=\nuW(\hx+(0,h))$, see \cite[Prop.~3.3b]{KR2015}. Hence
\begin{equation*}
\nuG
=
\piG_*(\nu)
=
\piG_*(\nuW(\hx+(0,h)))
=
\piG_*(\nuW(\hx)+(0,h))
=
\piG_*(\nuW(\hx))\ ,
\end{equation*}
where we used \eqref{eq:CW-inv} for the third identity.\\
d) In view of assertion c), $\Gamma^{-1} \{\rho(\hx)\}$ is a singleton for each $\rho(\hx)\in\rho(\CW)$. For countable acting groups $G$ it is well known that this implies that $\Gamma$ is almost \oneone. For uncountable groups we could not locate such a statement in the literature. So we provide a proof for the convenience of the reader:
In view of c) we only need to show that $\rho(\CW)$ is a dense $G_\delta$-set in 
$\hX/\pihX(\cH_{\inn(W)})$:
$\CW$ is a dense $G_\delta$-set in $\hX$ by \cite[Prop.~3.3c]{KR2015}.
Hence denseness of $\rho(\CW)$ follows as $\rho$ is continuous and onto.
In \eqref{eq:CW-inv} we argued that
$\CW$ - and hence also $\hX\setminus \CW$ - are invariant under translations by elements from the subgroup $H_W$.
Hence
$\rho(\hX\setminus \CW)=(\hX/\pihX(\cH_{\inn(W)}))\setminus\rho(\CW)$.
Indeed, suppose $\rho(\hx)=\rho(\hy)$ for some $\hx\in \CW$ and $\hy\in\hX \setminus\CW$. Then $\hy\in \hx+\cH_{\inn(W)}=\hx+\cH_W\subseteq \CW$, a contradiction.
To show that $\rho(\CW)$ is a $G_\delta$-set, it now suffices to show that $\rho(\hX\setminus\CW)$ is an $F_\sigma$-set. But this is obvious since $\hX\setminus\CW$ is an $F_\sigma$-set by Eqn.~\eqref{eq:CW-def} and hence a countable union of compact sets, and since $\rho$ is continuous.

\end{proof}

\begin{proof}[Proof of Theorem~\ref{theo:inn-periodic}]
a) Let $M$ be any non-empty, closed $S$-invariant subset of $\MWG$.
$(\hXprime,\hTprime)$ is a factor of $(M,S)$ by Lemma~\ref{lem:factor}. We prove that it is the maximal equicontinuous factor of $(M,S)$.

Let $W_0:=\overline{\inn(W)}$ and 
$H_0:=H_{W_0}=H_{\inn(W)}$ (see Lemma~\ref{lem:aper} for the second identity).
Then $W'=\varphi_{W_0}(W_0)$ is an aperiodic subset of $H'=H/H_{W_0}$ 
by Lemma~\ref{lemma:Wprime-aperiodic}.
As $W_0$ is topologically regular, also $W'$ is topologically regular (Lemma~\ref{lemma:Wprime}). Hence also $\inn(W')$ is aperiodic.

Let $\hXprime=\hX/\pihX(\cH_{\inn(W)})$ with induced $G$-action $\hTprime$, and recall from Corollary~\ref{coro:iso-quotients} that $\hXprime$ is isomorphic to $(G\times H')/\LL'$.  Theorem~\ref{theo:inn-aperiodic}
implies that $(\MWGprime,S)$ is an almost \oneone extension of 
$(\hXprime,\hTprime)$.
As $\MWGprime=\piG_*(\MWprime)=\piG_*(\MWzero)=\MWGzero$ by Proposition~\ref{prop:same-images} (applied to $W_0$ instead of $W$), also $(\MWGzero,S)$ is an almost \oneone extension of $(\hXprime,\hTprime)$.
It follows that also the unique minimal subsystem $(\overline{\nuWGzero(\CWzero)},S)$ of 
$(\MWGzero,S)$ is an almost \oneone extension of 
$(\hXprime,\hTprime)$.

Note next that 
$\CW\subseteq\CWzero$ and $\nuWGzero|_{\CW}=\nuWG|_{\CW}$, because $\partial W_0\subseteq\partial W$. As $\CW$ is dense 
in $\hX$ (and as $\CWzero$ is the set of continuity points of $\nuWGzero$), this implies
$\overline{\nuWGzero(\CWzero)}=\overline{\nuWG(\CW)}=\MminG$. Hence the minimal system 
$(\MminG,S)$ 
is an almost \oneone extension of $(\hXprime,\hTprime)$, so that $(\hXprime,\hTprime)$ is the maximal equicontinuous factor of $(\MminG,S)$.

Suppose now that $(\tilde{X},\tilde{T})$ is an equicontinuous factor of $(M,S)$
and observe that $\MminG\subseteq M$. Then the restriction of the factor map to $\MminG$ defines a factor map from
$(\MminG,S)$ to $(\tilde{X},\tilde{T})$. It follows that $(\tilde{X},\tilde{T})$ is a factor of $(\hXprime,\hTprime)$. As this holds for any equicontinuous factor $(\tilde{X},\tilde{T})$ of $(M,S)$, the system $(\hXprime,\hTprime)$ is in fact the maximal equicontinuous factor of $(M,S)$.\\
b) This is Lemma~\ref{lem:factor}d.
\end{proof}

\begin{remark}
A proof of Theorem~\ref{theo:inn-periodic} only for the case when $H_{\inn(W)}=H_W$ is much simpler: In that case $H'=H/H_W$, and $W'=\varphi_W(W)$ is aperiodic by Lemma~\ref{lemma:Wprime-aperiodic}. As $\MWG=\MWGprime$ by Proposition~\ref{prop:same-images}a,
all assertions of Theorem~\ref{theo:inn-periodic} follow from Theorem~\ref{theo:inn-aperiodic} applied to the cut-and-project scheme $(G,H',\LL')$ with window~$W'$.
\end{remark}

\begin{proof}[Proof of Theorem~\ref{theo:main-2-periodic}]
Assume first that the window $W$ is Haar regular. Then $W=W_{reg}=W_{\QM}$ by Corollary~\ref{coro:WQM=Wreg}. Let $W'=\varphi_W(W)$. This set is Haar regular by Lemma~\ref{lemma:Wprime} and aperiodic by Lemma~\ref{lemma:Wprime-aperiodic}. Hence $W'$ is also Haar aperiodic, see Remark~\ref{remark:Haar-and-other-periods}.
As $(\MWG,\QMG,S)=(\MWGprime,\QMGprime,S)$ by Proposition~\ref{prop:same-images}, the claim of the theorem follows now from Theorem~\ref{theo:main-2}.

In the general case, note that $(\MW,\QM,S)$ is measure-theoretically isomorphic to $(\cM,\QM,S)$. As the present theorem applies to the Haar regularized window $W_{reg}$,
we must only show that $\QM=m_\hX\circ (\nuW)^{-1}$ equals $\QMreg=m_\hX\circ (\nuWreg)^{-1}$ on $\cM$. But this follows from the observation that
\begin{equation*}
\left\{\hx\in\hX:\nuW(\hx)\neq\nuWreg(\hx)\right\}
\subseteq
\pihX\left(\bigcup_{\ell\in\LL}\left((G\times(W\setminus W_{reg}))-\ell\right)\right),
\end{equation*}
and this is a set of $m_\hX$-measure zero, because $\LL$ is countable and $m_H(W\setminus W_{reg})=0$.
\end{proof}

\begin{proof}[Proof of Theorem~\ref{theo:main-periodic}]
Let $\PG$ be an ergodic  $S$-invariant probability measure on $\MWG\setminus\{\0\}$. Take any ergodic $S$-invariant probability measure $P$ on $\MW\setminus\{\0\}$ satisfying $\PG=P\circ(\piG_*)^{-1}$. In particular we have $W_P\ne\emptyset$.  Let $\hXprime=\hX/\pihX(\cH_{W_{P}})$ with induced $G$-action $\hTprime$ and Haar measure $m_{\hXprime}$. 
Denote
\begin{equation*}
\cA_P:=\left\{\nu\in\MW: \SH(\nu)=W_P\right\},
\end{equation*}
where $W_P\subseteq W$ is the Haar regular set from Lemma~\ref{lemma:SH-inv-ergod}, for which $\SH(\nu)=W_P$ for $P$-a.a.~$\nu$. Hence $P(\cA_P)=1$, and $\cA_P$ is Borel measurable and $S$-invariant, because $\SH$ is (Lemmas~\ref{lem:lsc} and~\ref{lemma:SH-inv}a). Then also $\piG_*(\cA_P)\subseteq\MWG\setminus\{\0\}$ is $S$-invariant, 
it is Borel measurable by Lemma~\ref{lemma:Borel-measurable-strong},
and $\PG(\piG_*(\cA_P))=P\left((\piG_*)^{-1}(\piG_*(\cA_P))\right)\geqslant P(\cA_P)=1$.

Denote by $\gamma$ the factor map from $\MW\setminus\{\0\}$ onto  $\hX$
\footnote{This is the map $\pihX_*\circ(\piGH_*)^{-1}$ from \cite[Thm.~1a]{KR2015}.\label{footnote:pi*}}, and by $\rho$ the factor map from $\hX$ onto $\hX/\pihX(\cH_{W_P})$.
We define now a map 
\begin{equation*}
\Gamma:\piG_*(\cA_P)\to\hX/\pihX(\cH_{W_P}),\quad
\left\{\Gamma(\nuG)\right\}:=(\rho\circ\gamma)\left((\piG_*)^{-1}\{\nuG\}\cap\cA_P\right).
\end{equation*}
In order to see that $\Gamma$ is well defined, observe first
the cardinality is at least $1$, because $\nuG\in\piG_*(\cA_P)$. 
On the other hand, if
$\nu,\nu'\in(\piG_*)^{-1}\{\nuG\}\cap\cA_P$, then $\gamma(\nu')-\gamma(\nu)\in\pihX(\cH_{W_P})$ by Lemma~\ref{lemma:technical}, i.e. $(\rho\circ\gamma)(\nu')=(\rho\circ\gamma)(\nu)$.

Next observe that $\Gamma$ commutes with the dynamics: For each $g\in G$,
\begin{equation*}
\begin{split}
\left\{\Gamma(S_g\nuG)\right\}
&=
(\rho\circ\gamma)\left((\piG_*)^{-1}(S_g\{\nuG\})\cap\cA_P\right)
=
(\rho\circ\gamma)\left(S_g\left((\piG_*)^{-1}\{\nuG\}\right)\cap S_g(\cA_P)\right)\\
&=
\widehat{T'_g}\left((\rho\circ\gamma)\left((\piG_*)^{-1}\{\nuG\}\cap\cA_P\right) \right)
=
\widehat{T'_g}\left\{\Gamma(\nuG)\right\}\\
&=
\left\{\widehat{T'_g}(\Gamma(\nuG))\right\}
\end{split}
\end{equation*}

It remains to show that $\piG_*(\cA_P)\cap\Gamma^{-1}(K')$ is Borel measurable for each closed subset $K'$ of $\hXprime$. Then it follows that $\Gamma$ is Borel measurable and $(\hXprime,m_{\hXprime},\hTprime)$ is a measure-theoretic factor of $(\MWG,\PG,S)$ as claimed in Theorem~\ref{theo:main-periodic}. 

So let $K:=(\rho\circ\gamma)^{-1}(K')\subseteq \MW$. Then $K$ is closed, and
\begin{equation}\label{eq:last-identity}
K\cap\cA_P=(\piG_*)^{-1}\left(\piG_*(K\cap\cA_P)\right)\cap\cA_P\ .
\end{equation}
The $\subseteq$-inclusion is trivial. To see the reverse inclusion, let $\nu\in\cA_P$ and assume that there exists $\nu'\in K\cap\cA_P$ such that $\piG_*(\nu)=\piG_*(\nu')$. Then $\gamma(\nu)-\gamma(\nu')\in\pihX(\cH_{W_P})$ by Lemma~\ref{lemma:technical}, so that 
$(\rho\circ\gamma)(\nu)\in(\rho\circ\gamma)(\nu')+\rho(\pihX(\cH_{W_P}))=(\rho\circ\gamma)(\nu')\in(\rho\circ\gamma)(K)\subseteq K'$, i.e.
$\nu\in(\rho\circ\gamma)^{-1}(K')=K$.

Now let $\nuG\in\piG_*(\cA_P)$. Then
\begin{equation*}
\begin{split}
\Gamma(\nuG)\in K'
&\Leftrightarrow
\left\{\Gamma(\nuG)\right\}\subseteq K'\\
&\Leftrightarrow
(\rho\circ\gamma)\left((\piG_*)^{-1}\{\nuG\}\cap \cA_P\right)\subseteq K'\\
&\Leftrightarrow
(\piG_*)^{-1}\{\nuG\}\cap \cA_P\subseteq (\rho\circ\gamma)^{-1}(K')= K\\
&\Leftrightarrow
(\piG_*)^{-1}\{\nuG\}\cap \cA_P\subseteq K\cap\cA_P\\
&\Leftrightarrow
\nuG\in\piG_*(K\cap\cA_P)\ .
\end{split}
\end{equation*}
The last equivalence is seen as follows:\\ ,,$\Rightarrow$'':
As $\nuG\in \piG_*(\cA_P)$, there exists some $\nu\in(\piG_*)^{-1}\{\nuG\}\cap\cA_P\subseteq F\cap\cA_P$, so that
$\nuG=\piG_*(\nu)\in\piG_*(F\cap\cA_P)$.\\
,,$\Leftarrow$'': Let $\nu\in(\piG_*)^{-1}\{\nuG\}\cap\cA_P$. Then $\piG_*(\nu)=\nuG\in\piG_*(F\cap\cA_P)$, so that in view of \eqref{eq:last-identity},
\begin{equation*}
\nu\in(\piG_*)^{-1}\left(\piG_*(F\cap\cA_P)\right)\cap\cA_P
=
F\cap\cA_P\ .
\end{equation*}

Hence $\piG_*(\cA_P)\cap\Gamma^ {-1}(K')=\piG_*(K\cap\cA_P)$, and this set is Borel measurable by Lemma~\ref{lemma:Borel-measurable-strong}.
\end{proof}

\section{Relatively compact windows}\label{sec:rcwin}

For suitable relatively compact windows $W\subseteq H$, dynamical properties are the same as in the compact case. This has already been observed in \cite[Rem.~3.16]{KR2015}. 

\subsection{Topological results}

For some topological results, we assume that the boundary $\partial W$ is nowhere dense. This condition characterises denseness  of the set $\CW$ of continuity points of the map $\nuW:\MW\to\hX$, see Proposition~\ref{lem:dGd}. As any compact window has nowhere dense boundary, this condition generalises the compact case. 
The next lemma shows that this condition also generalises topological regularity.
\begin{lemma}\label{lemma:bdW-nowhere-dense}
For each $W\subseteq H$,  $\partial W$ is nowhere dense if and only if 
$\inn(\overline{\inn(W)})=\inn(\overline{W})$.
\end{lemma}
\begin{proof}
As $\partial W=\overline{W}\cap\overline{W^c}$, we have $\inn(\partial W)=\inn(\overline{W})\cap\inn(\overline{W^c})$. Hence $\partial W$ is nowhere dense if and only if $\inn(\overline{W})\subseteq(\inn(\overline{W^c}))^c=\overline{\inn(W)}$. But this is obviously equivalent to 
$\inn(\overline{\inn(W)})=\inn(\overline{W})$.
\end{proof}

The next lemma generalises slightly Lemma 6.1 in~\cite{KR2015}.

\begin{lemma}\label{lem:charcw}
Let $W\subseteq H$ be relatively compact. Then the set $\CW\subseteq \hX$ of continuity points of the map $\nuW:\hX\to \MW$ is given by
\begin{displaymath}
\CW=\pihX\left(\bigcap_{\ell\in\LL}\left((G\times(\partial W)^c)-\ell\right)\right) \ .
\end{displaymath}
\end{lemma}

\begin{proof}
This is seen by reinspecting the proof of \cite[Lem.~6.1]{KR2015}. If $x_H\not\in\bigcap_{\ell\in \LL}((\partial W)^c-\ell_H)$, then there is $\ell\in \LL$ such that $x_H+\ell_H\in\partial W$. In the case $x_H+\ell_H\in W$ we take $x^n_H\in (W)^c-\ell_H$ such that $x_H^n\to x_H$ as $n\to\infty$ and let $x^n=(x_G,x_H^n)$. Then, for each sufficiently small open neighbourhood $U$ of $x+\ell$ in $G\times H$ and sufficiently large $n$ we have $\nuW(x^n+\LL)(U)=1_{W}(x_H^n+\ell_H)\cdot\delta_{x_n+\ell}(U)=0$ while 
$\nuW(x+\LL)(U)=1_{W}(x_H+\ell_H)\cdot\delta_{x+\ell}(U)=1$ so that $\nuW(x^n+\LL)\not\to\nuW(x+\LL)$. In particular, $(x+\LL)\not\in \CW$. If $x_H+\ell_H\in (W)^c$, we may argue analogously with a sequence $x^n_H\in W-\ell_H$ such that $x_H^n\to x_H$ as $n\to\infty$. 

Conversely, let $x+\LL,x^n+\LL\in\hX$, $\lim_{n\to\infty}(x^n+\LL)=(x+\LL)$ and take a sufficiently small open neighbourhood $U$ of $x+\ell$ in $G\times H$. Assume that $x_H\in\bigcap_{\ell\in \LL}((\partial W)^c-\ell_H)$, i.e., $(x+\ell)_H\not\in\partial W$ for all $\ell\in \LL$. Now consider $\ell\in\LL$ such that $x_H+\ell\in W$. Then by assumption $x_H+\ell_H\in \inn(W)$, which implies  $1=\nuW(x+\LL)(U)=\nuW(x^n+\LL)(U)$ for sufficiently large $n$. If $\ell\in\LL$ such that $x_H+\ell_H\notin W$, then by assumption $x_H+\ell\in (\overline{W})^c$, which implies  $0=\nuW(x+\LL)(U)=\nuW(x^n+\LL)(U)$ for sufficiently large $n$. 
This implies $\lim_{n\to\infty}\nuW(x^n+\LL)=\nuW(x+\LL)$ and hence $x+\LL\in \CW$.
\end{proof}

The following proposition collects some further properties of the set $\CW$.
\begin{proposition}\label{lem:dGd}
Let $W\subseteq H$ be relatively compact. Then the set $\CW\subseteq \hX$ of continuity points of the map $\nuW:\hX\to \MW$ is a $G_\delta$-set.  The set $\CW$ is dense in $\hX$ if and only if $\partial W$ is nowhere dense in $H$.
Otherwise $\CW=\emptyset$.
\end{proposition}

\begin{proof}\footnote{For compact windows this was claimed in \cite[Prop.~3.3c]{KR2015}. As the proof of that proposition contains a mistake, we provide a full proof for the more general case treated here. (Indeed, the argument given in footnote $\dagger$ of the proof of that proposition is wrong, because, with the sets $U_\ell$ and $X$ defined there, it is not true that $\pihX(X\cap U_\ell)=\pihX(U_\ell)$.)}
Let $A:=\bigcap_{\ell\in\LL}\left((G\times(\partial W)^c)-\ell\right)$ and $B:=(G\times H)\setminus A=\bigcup_{\ell\in\LL}G\times(\partial W-\ell_H)$,
and recall that $\hX=(G\times H)/\LL$.
As both sets, $A$ and $B$, are invariant under translations by elements from the lattice $\LL$, 
$\pihX(B)=\hX\setminus\pihX(A)$.
Indeed, suppose $\pihX(a)=\pihX(b)$ for some $a\in A$ and $b\in B$. Then $b\in a+\LL\subseteq A$, a contradiction.
To show that $\pihX(A)$ is a $G_\delta$-set, it suffices to show that $\pihX(B)$ is a $F_\sigma$-set. To that end recall that $G$ is $\sigma$-compact, i.e. there are compact $K_1,K_2,\ldots\subseteq G$ such that $G=\bigcup_{j\in\mathbbm N}K_j$. Hence 
$\pihX(B)=\bigcup_{\ell\in\LL}\bigcup_{j\in\mathbbm N}\pihX\left(K_j\times(\partial W-\ell_H)\right)$ is a countable union of compact sets, hence $F_\sigma$.

Now assume that $\partial W$ is nowhere dense. Then $A_H:=\bigcap_{\ell\in\LL}\left((\partial W)^c-\ell_H\right)$ is dense in $H$ as $\partial W$ is nowhere dense and as  $H$ is a Baire space. This readily implies that
$A=(\piH)^{-1}(A_H)=G\times A_H$ is dense in $G\times H$. Now denseness of $\pihX(A)=\CW$ in $\hX$ follows as $\pihX$ is continuous and onto. 

Conversely, assume that $\partial W$ is not nowhere dense. Then there is some nonempty open set $O\subseteq \partial W$, which implies
$A_H^c\supseteq\bigcup_{\ell\in\LL}(O-\ell_H)=H$ because of the denseness of $\piH(\LL)$ in $H$. Thus
$A_H=\emptyset$ and $\CW=\pihX((\piH)^{-1}(A_H))=\emptyset$, in particular $\CW$ is not dense.
\end{proof}

Let us denote by $\GnuW:=\{(\hx, \nuW(\hx)):\hx\in\hX\}\subseteq \hX\times \MW$ the graph of the map $\nuW$, and by $\GMW$ its closure in the vague topology. Likewise, denote by $\cG(\nuW|_{\CW})$ the restriction of $\GnuW$ to its continuity points. We have the following general result on minimal subsets of $\GMW, \MW$ and $\MWG$.
\begin{lemma}\label{lemma:minimal-subset}
Let $W\subseteq H$ be relatively compact. Then
\begin{compactenum}[a)]
\item 
The set $\overline{\cG(\nuW|_{\CW})}$ is the unique minimal subset of $\GMW$. 
\item The set $\overline{\nuW(\CW)}$ is the unique minimal subset of $\MW$.
\item The set $\overline{\nuWG(\CW)}$ is the unique minimal subset of $\MWG$.
\item If $\inn(W)=\emptyset$, then $\overline{\nuW(\CW)}$ and $\overline{\nuWG(\CW)}$ are singletons consisting of the zero measure only.
\end{compactenum}
\end{lemma}

\begin{proof}
This can be seen by re-inspecting the proofs in \cite{KR2015}, but we give a simple direct argument for the ease of the reader. \\
a) Let $\emptyset \ne A\subseteq \GMW$ be any closed invariant set. Then $\emptyset \ne \pihX(A)\subseteq \hX$ is closed  invariant. Hence $\pihX(A)=\hX\supseteq \CW$, since $(\hX,\hat T)$ is minimal. 
As $(\pihX)^{-1}\{\hx\}\cap \GMW=\{(\hx,\nuW(\hx))\}$ for each $\hx\in \CW$, this implies $A\supseteq \cG(\nuW|_{\CW})$, which means
 $A\supseteq \overline{\cG(\nuW|_{\CW})}=:A_{min}$.\\
b) Let $\emptyset\ne B\subseteq \MW$ be any closed invariant set. Then $\emptyset \ne (\piGH_*)^{-1}(B)\subseteq  \GMW$ is closed invariant. By the previous result, we infer $(\piGH_*)^{-1}(B)\supseteq A_{min}$. 
Hence $B\supseteq \piGH_*(A_{min})\supseteq{\nuW(\CW)}$, so that
$\overline{\nuW(\CW)}\subseteq\overline{B}= B$.
\\
c) This follows using the same argument as in b).  \\
d) If $\inn(W)=\emptyset$, then $(\partial W)^c\subseteq (W)^c$. Hence $\hx=x+\LL\in \CW$ implies  $x_H\in \bigcap_{\ell\in\LL}\left((\partial W)^c-\ell_H\right)\subseteq   \bigcap_{\ell\in\LL}\left((W)^c-\ell_H\right)$. But the latter condition means $(x+\LL)\cap(G\times W)=\emptyset$. Hence the claim follows.
\end{proof}

Consider now the window $\overline{W}$, as well. We infer from \cite[Lem.~5.4]{KR2015} that, for each $\nu\in\MW\setminus\{\0\}$, there is a unique
$\hat\pi(\nu)\in\hX$ such that $\supp(\nu)\subseteq\supp(\nuWbar(\hat\pi(\nu)))$. Thus the map $\hat\pi: \MW\setminus\{\0\}\to \hX$ is still well-defined and continuous in our more general setting, and it satisfies $\hat\pi=\pihX_*\circ (\piGH_*)^{-1}$. We have the following version of \cite[Thm.~1a]{KR2015}.

\begin{proposition}
Let $W\subseteq H$ be relatively compact and such that $\partial W$ is nowhere dense in $H$. Assume that $\inn(\overline{W})$ is nonempty (being equivalent to $\inn(W)$ nonempty, in this case). Then
\begin{compactenum}[a)]
\item $\pihX_*:(\cG\MW,S)\to(\hX,\hT)$ is a topological almost \oneone-extension of its maximal equicontinuous factor.
\item $\hat\pi:(\MW,S)\to(\hX,\hT)$ is a topological almost \oneone-extension of its maximal equicontinuous factor.
\end{compactenum}
\end{proposition}

\begin{proof}
a) We can argue as in the proof of \cite[Prop.~3.5c]{KR2015}.  The assumption $\partial W$ nowhere dense guarantees that $\CW$ is a dense $G_\delta$-set by Proposition~\ref{lem:dGd}.\\
b) This follows from a) by noting that  the statements and proofs of \cite[Prop.~3.5b]{KR2015} and \cite[Prop.~3.3e]{KR2015} still apply to the present situation. 
\end{proof}

If $\partial W$ is nowhere dense, then $\MW$ and $\MWbar$ have the same unique minimal subset, and a similar result holds for the $G$-projections.
\begin{lemma}(See also \cite[Cor.~1b and Remark 3.16]{KR2015})\label{lemma:same-Min}
Let $W\subseteq H$ be relatively compact and such that $\partial W$ is nowhere dense in $H$. Then $\overline{\nuW(\CW)}=\overline{\nuWbar(\CWbar)}$
and $\overline{\nuWG(\CW)}=\overline{\nuWGbar(\CWbar)}$.
\end{lemma}
\begin{proof}
As $\partial \overline{W}\subseteq\partial W$, we have $\CW\subseteq\CWbar$ and
$\nuW|_{\CW}=\nuWbar|_{\CW}$, because $1_{\overline{W}}(h)=1_{W}(h)$ for all $h\in H\setminus\partial W$. Hence
$\overline{\nuW(\CW)}=\overline{\nuWbar(\CW)}\subseteq\overline{\nuWbar(\CWbar)}$.
On the other hand, as $\CW$ is dense in $\hX$ and as $\nuWbar$ is continuous on $\CWbar$, we have $\nuWbar(\CWbar)\subseteq\overline{\nuW(\CW)}$. This proves the first identity. 
The second identity follows at once, because $\nuWGbar=\piG_*\circ\nuWbar$ with a continuous $\piG_*$.
\end{proof}

Now we are ready to state and prove
 the following extensions of Theorems~\ref{theo:inn-aperiodic} and~\ref{theo:inn-periodic}.

\begin{theoremIprime}\label{theo:inn-aperiodicprime}
Let $W\subseteq H$ be relatively compact and $\partial W$ be nowhere dense. Assume that $\inn(\overline{W})$ is aperiodic (so in particular non-empty). 
\begin{compactenum}[a)]
\item The topological dynamical systems $(\MW,S)$ and $(\MWG,S)$ are isomorphic, and both are almost \oneone extensions of their maximal equicontinuous factor $(\hX,\hT)$.
\item Denote by $\Gamma:\MWG\to\hX$ the factor map from a).
If $M$ is a non-empty, closed $S$-invariant subset of $\MWG$, then $(M,S)$ is an
almost \oneone extension of its maximal equicontinuous factor $(\hX,\hT)$ with factor map 
$\Gamma|_{M}$.
\end{compactenum}
\end{theoremIprime}

\begin{proof}[{\bf \em Sketch of proof of Theorem~\ref{theo:inn-aperiodicprime}}]
In Section~\ref{sec:proof-1}, the space $\widetilde{\MWbar}$ was defined such that 
$\MW\subseteq\widetilde{\MWbar}$. Hence
 Lemma~\ref{lem:lsc} applies also to $\SH|_{\MW\setminus\{\0\}}$. 
In view of Lemma~\ref{lemma:same-Min}, 
Lemma~\ref{lemma:SH-inv}c) and d) and Lemma~\ref{lemma:shifted-versions} remain valid.
Keeping in mind the above results, one readily checks that the proof of  Theorem~\ref{theo:inn-aperiodic} also applies under the above assumptions.
\end{proof}

\begin{theoremIIprime}\label{theo:inn-periodicprime}
Let $W\subseteq H$ be relatively compact and $\partial W$ be nowhere dense. 
Assume that $\inn(\overline{W})\neq\emptyset$.
Let $\hXprime=\hX/\pihX(\cH_{\inn(\overline{W})})$ with induced $G$-action $\hTprime$, 
and let $M$ be any non-empty, closed $S$-invariant subset of $\MWG$ (thus including the case $M=\MWG$).
\begin{compactenum}[a)]
\item $(\hXprime,\hTprime)$ is the maximal equicontinuous factor of
the topological dynamical system $(M,S)$.
\item If $H_{\inn(\overline{W})}=H_{\overline{W}}$, then $(M,S)$ is an almost \oneone extension of $(\hXprime,\hTprime)$.
\end{compactenum}
\end{theoremIIprime}

\begin{proof}[{\bf \em Sketch of proof of Theorem~\ref{theo:inn-periodicprime}}]
Here we note that $H_{\inn(\overline{W})}$ is compact due to Lemma~\ref{lem:aper}c) and d). 
This ensures that all arguments in the proof
of Theorem~\ref{theo:inn-periodic} for compact windows directly apply to the present situation.
\end{proof}

\subsection{Measure-theoretic results}\label{subsec:mtr}

For measure-theoretic results, let us assume that $W\subseteq H$ is relatively compact and measurable.
In that situation, the map $\nuW:\hX\to\MW\subseteq\cM$ is still measurable such that the Mirsky measure is well defined, compare \cite[Rem.~3.16]{KR2015}. In fact Propositions~\ref{prop:MWfactor}, \ref{prop:factor} and \ref{prop:mirsky} continue to hold. In particular, $(\MW,\QM,S)$ is a measure-theoretic factor of $(\hX,m_\hX,\hT)$, and thus the same holds for $(\MWG,\QMG,S)$. Hence both systems have pure point dynamical spectrum.

\medskip

Whereas the statement of Proposition~\ref{prop:MWfactor} is obvious from measurability of $\nuW$, we give proofs of the other propositions for the convenience of the reader.

\begin{proof}[Proof of Proposition~\ref{prop:factor}]
Note that $P\circ \hat\pi^{-1}$ is a probability measure on $\hX$ by assumption on $P$. As $P\circ \hat\pi^{-1}$ is $\hT$-invariant and the $\hT$-action is uniquely ergodic, it 
thus equals $m_\hX$, compare Fact \ref{en:ass}\,(\ref{item2.2:4}). In particular,
$m_\hX(\nuW^{-1}\{\0\})=P\,\big((\nuW\circ\hat\pi)^{-1}\{\0\}\big)=0$, because 
$\supp(\nu)\subseteq\supp(\nuW(\hat\pi(\nu))$ for all $\nu\in\MW\setminus\{\0\}$.
But this implies $m_H(W)>0$ \cite[Prop.~3.6b]{KR2015}, which continues to hold in the non-compact setting.  The assertion of the proposition now follows as the desired factor map \cite[Def.~2.7]{EinWard2011} is provided by $\hat \pi$.
\end{proof}

\begin{proof}[Proof of Proposition~\ref{prop:mirsky}]
As $m_H(W)>0$, we have $m_\hX(\nuW^{-1}\{\0\})=0$ by \cite[Prop.~3.6b]{KR2015}, which continues to hold in the non-compact setting. Hence we have $\QM(\MW\setminus \{\0\})=m_\hX\circ(\nuW)^{-1}(\MW\setminus \{\0\})=1$. We thus can combine the statements of Proposition~\ref{prop:factor} and Proposition~\ref{prop:MWfactor} to get the result.
\end{proof}

For general relatively compact measurable windows $W$, it might be difficult to give an isomorphism between $(\MWG,\QMG,S)$ and an explicit group rotation. But the previous results for compact windows, i.e., Theorems~\ref{theo:main-2} and~\ref{theo:main-2-periodic}, continue to hold
for windows $W$ that are compact modulo $0$. Consider a window $W$ that is compact modulo $0$ and its Haar regularization $W_{reg}=\supp((m_H)|_{W})$.
Then the Mirsky measures $\QM$ and $\QMreg$ coincide on $\mathcal M$, because
\begin{equation*}
\left\{\hx\in\hX:\nuWreg(\hx)\neq\nuW(\hx)\right\}
\subseteq
\pihX\left(\bigcup_{\ell\in\LL}\left((G\times(W_{reg}\triangle W))-\ell\right)\right)\ ,
\end{equation*}
which is a set of $m_\hX$-measure zero as $\LL$ is countable. This implies that $(\MW,\QM,S)$ is measure-theoretically isomorphic to $(\MWreg,\QMreg,S)$. 
As the Haar periods of $W_{reg}$ coincide with those of $W$,
Theorems~\ref{theo:main-2} and~\ref{theo:main-2-periodic}, which apply to $W_{reg}$, continue to hold for $W$.

\medskip

In oder to better understand the passage from the extended hull $\MWG$ to the usual hull $\MWG(\hx)$ for model sets without compact windows, we  discuss the maximal density condition in that case.

\begin{remark}(Generic configurations)\label{rem:genconf}\\
Assume that some configuration $\nuWG(\hx)$ has maximal density in the sense that $d(\nuWG(\hx))=\mathrm{dens}(\LL)\cdot m_H(\overline{W})$ along some given tempered van Hove sequence $(A_n)_n$, compare also Section~\ref{sec:proof-2-3}. 
In that case $\nuWG(\hx)$ is generic for the Mirsky measure $\QMGbar$
on $\mathcal M^{\scriptscriptstyle G}$, as will be shown below. This implies that $(\MWG(\hx), \QMGbar, S)$ is isomorphic to $(\widetilde{\MWGbar}, \QMGbar, S)$, where $\widetilde{\MWGbar}=\{\nu\in\cMG: \nu\leqslant\nuWGbar(\hx)\text{ for some }\hx\in\hX\}$, compare the proof of \cite[Thm.~5]{KR2015}.

To see genericity, note that  for any $\hx\in\hX$ we have $0\leqslant\nuWG(\hx)\leqslant\nuWGbar(\hx)$, so in particular the density of $\nuWG(\hx)$ along $(A_n)_n$ is bounded by that of $\nuWGbar(\hx)$ and hence by $\mathrm{dens}(\LL)\cdot m_H(\overline{W})$, see \cite[Thm.~3]{KR2015}.
If $\nuWG(\hx)$ achieves this maximal density, then 
the density of $\nuWGbar(\hx)-\nuWG(\hx)$ is clearly zero. 
Now consider
for $\nuWG(\hx)\in\MWG$ its empirical measures
\begin{displaymath}
Q^{G,n}_{W,\hx}
:=
\frac{1}{m_G(A_n)}\int_{A_n}\delta_{S_{g}\nuWG(\hx)}\,dm_G(g)\ , 
\end{displaymath}
and likewise for $\nuWGbar(\hx)\in \MWGbar$ its empirical measures $Q^{G,n}_{\overline{W},\hx}$ .
The previous reasoning shows that the empirical measures
$Q_{W, \hx}^{G,n}$
 are asymptotically equivalent to the measures $Q_{\overline{W}, \hx}^{G,n}$ 
in the sense that both sequences do have the same weak limit points, and \cite[Thm.~5]{KR2015} implies that the measures $Q_{W, \hx}^{G,n}$ 
converge weakly to $\QMGbar$. 
It follows that statistical properties of $\nuWG(\hx)$ and $\nuWGbar(\hx)$, like pattern frequencies and especially their autocorrelation coefficients, coincide for such $\hx$, and that they are determined by the Mirsky measure $\QMGbar$, to which Theorems~\ref{theo:main-2} and~\ref{theo:main-2-periodic} apply.

Note, however, that for $m_{\hX}$-a.a $\hx$ the configuration $\nuW(\hx)$ has density
$\mathrm{dens}(\LL)\cdot m_H({W})$, see \cite[Thm. 1]{Moody2002}. Thus the above reasoning applies to $\QMG$-Mirsky-typical configurations only, if $m_H(W)=m_H(\overline{W})$, i.e., if $W$ is compact modulo $0$.

\end{remark}

\begin{remark}(Autocorrelation measure, diffraction spectrum and generic configurations) \\
For $\PG\in \mathcal{M}_S(\widetilde{\MWGbar})$, the set of $S$-invariant probability measures on $\widetilde{\MWGbar}$, the so-called autocorrelation measure $\gamma_{\PG}$ is a positive definite Radon measure on $G$ that is naturally associated to $\PG$. In particular, $\gamma_{\PG}$ is Fourier transformable. 
We recall its definition from \cite[Prop.~6]{BaakeLenz2004}, \cite[Sec.~2.3]{LenzRichard07} and \cite[Lemma~4.1]{LenzStrungaru09}. Take any $\psi\in C_c(G)$ such that $m_G(\psi)=1$ and define the Radon measure $\gamma_{\PG}$ via its associated linear functional $\gamma_{\PG}:C_c(G)\to\mathbb C$ by requiring
\begin{displaymath}
\gamma_{\PG}(\varphi):= 
\int_{\widetilde{\MWGbar}} 
\left(\int_G {\int_G {\varphi(s-t)\psi(t)} d \nu(s)} d \nu(t) \right)
d \PG(\nu) 
\end{displaymath}
for every $\varphi\in C_c(G)$. 
The measure $\gamma_{\PG}\in \cMG$ is independent of the choice of $\psi$, and it is a rather direct consequence of the above definition that the map $\gamma: \mathcal{M}_S(\widetilde{\MWGbar})\to\cMG$, defined by $\PG \mapsto \gamma_{\PG}$, is continuous with respect to the vague topologies.
For $\PG:= \QMG$, the Mirsky measure on $\widetilde{\MWGbar}$\,, the Fourier transform $\widehat{\gamma_{\QMG}}$ is a pure point measure as
$(\widetilde{\MWGbar}, \QMG,S)$ has pure point dynamical spectrum \cite[Thm.~7]{BaakeLenz2004}. One says that $(\widetilde{\MWGbar}, \QMG,S)$ has pure point diffraction spectrum.

Now fix any configuration $\nuG\in \widetilde{\MWGbar}$ and consider its empirical measures $Q^{n}_{\nuG}$ on $\widetilde{\MWGbar}$, given as
\begin{displaymath}
Q^{n}_{\nuG}
:=
\frac{1}{m_G(A_n)}\int_{A_n}\delta_{S_{g}\nuG}\,dm_G(g) \ , 
\end{displaymath}
along some fixed tempered van Hove sequence $(A_n)_n$ in $G$. As $\widetilde{\MWGbar}$ is compact, we can assume that the associated sequence $(\gamma_{Q_{ \nuG}^{n}})_n$ of empirical autocorrelations converges to a limit $\gamma_{\nuG}$, possibly after passing to some subsequence of $(A_n)_n$. For sufficiently large $n$,  the Fourier transform of  $\gamma_{\nuG}$  describes the outcome of a diffraction experiment on a physical realisation of $\nuG$ restricted to $A_n$. A standard tedious calculation which we omit yields the explicit expression $\gamma_{\nuG}=\sum_{\ell \in \LL}\eta(\ell)\delta_{\ell_G}$, where
\begin{displaymath}
\eta(\ell)=\lim_{n\to\infty} \frac{1}{m_G(A_n)} \mathrm{card}\big(\mathrm{supp}(\nuG)\cap (\ell_G+\mathrm{supp}(\nuG))\cap A_n\big) \ge 0
\end{displaymath}
are the autocorrelation coefficients of $\gamma_{\nuG}$.
In particular, if $\nuG$ is generic for the Mirsky measure $\QMG$ on $\widetilde{\MWGbar}$, i.e., if $(Q_{\nuG}^{n})_n$ converges to  $\QMG$, then by continuity of $\gamma$ we have  $\gamma_{\nuG}=\gamma_{\QMG}$. By the above argument, this implies that $\nuG$ is pure point diffractive. In particular,  from \cite[Eqn.~(18)]{KR2015} we obtain explicit expressions for the autocorrelation coefficients, namely $\eta(\ell)=\mathrm{dens}(\LL)\cdot m_H(W\cap (W+\ell_H))$ for all $\ell$ such that $\eta(\ell)\ne0$.
 
The above result alternatively follows by combining Theorem 5 and Proposition 6 in \cite{BaakeLenz2004}. That there is a  full $\QMG$-measure set of configurations $\nuWG(\hx)\in (\nuWG)^{-1}(\hX)\subseteq \MWG$ such that $\gamma_{\nuWG(\hx)}=\gamma_{\QMG}$ on any given tempered van Hove sequence $(A_n)_n$ has been shown by Moody \cite[Cor.~1]{Moody2002}, by constructing a full $m_\hX$-measure set of pure point diffractive $\nuWG(\hx)$ via repeated applications of Birkhoff's ergodic theorem.

A sufficient criterion for pure point diffractiveness has been observed in \cite{BaakeHuckStrungaru15, KR2015}: The configuration $\nuWG(\hx)$ is pure point diffractive if $\nuWG(\hx)$ has maximal density $d(\nuWG(\hx))=\mathrm{dens}(\LL)\cdot m_H(\overline{W})$. Indeed, in that case $\nuWG(\hx)$ is generic for the Mirsky measure $Q_{\scriptscriptstyle \overline{W}}^{\scriptscriptstyle G}$ on $\widetilde{\MWGbar}$ by Remark~\ref{rem:genconf} above. Note however that such $\nuWG(\hx)$ is typically  not generic for the Mirsky measure $\QMG$ on $\widetilde{\MWGbar}$.

\end{remark}

\end{document}